\documentclass[a4paper,11pt]{article}

\usepackage{amsmath, amscd}

\usepackage{amsthm}

\usepackage{tikz}
\usetikzlibrary{3d}
\usepackage{tikz-cd}
\usepackage{amssymb}
\usepackage{latexsym}
\usepackage{enumerate}

\newtheorem{Theorem}{Theorem}[section]

\newtheorem{Definition}[Theorem]{Definition}
\newtheorem{Proposition}[Theorem]{Proposition}
\newtheorem{Lemma}[Theorem]{Lemma}
\newtheorem{Corollary}[Theorem]{Corollary}
\newtheorem{Remark}[Theorem]{Remark}

\newtheorem{Example}[Theorem]{Example}

\begin{document}

\title{Deforming reducible representations of surface and 2-orbifold groups}

\author{Joan Porti\footnote{Partially supported by 
FEDER-AEI (grant numbers PID2021-125625NB-100 and
Mar\'\i a de Maeztu Program CEX2020-001084-M)
}} 

\date{\today}
\maketitle

\begin{abstract}
For a compact 2-orbifold with negative Euler characteristic $\mathcal O^2$, 
the variety of characters of $\pi_1(\mathcal O^2)$ in $\mathrm{SL}_{n}(\mathbb R)$
is a non-singular manifold at  $\mathbb C$-irreducible representations. In this paper we prove that when 
a $\mathbb C$-irreducible representation of $\pi_1(\mathcal O^2)$ in $\mathrm{SL}_{n}(\mathbb R)$ is viewed
in  $\mathrm{SL}_{n+1}(\mathbb R)$, then the variety of characters is singular,
and we describe the singularity. 
\end{abstract}

\tableofcontents
 
\section{Introduction}
\label{Section:introduction}

Let $\mathcal O^2$ be a compact 2-orbifold with $\chi(\mathcal O^2)<0$, 
possibly with boundary.
The group $
\mathrm{PSL}_n(\mathbb R) $ acts by conjugation 
algebraically on the variety of representations
of $\pi_1(\mathcal O^2)$ in $\mathrm{SL}_{n}(\mathbb R)$, and its quotient in 
real invariant theory is called the variety of characters
$$X( \mathcal O^2, \mathrm{SL}_{n}(\mathbb R))
=\hom( \pi_1(\mathcal O^2) , \mathrm{SL}_{n}(\mathbb R) ) /\!/  \mathrm{PSL}_n(\mathbb R). 
$$
Along the paper we shall assume $n\geq 2$.

A representation $\pi_1(\mathcal O^2)\to \mathrm{SL}_{n}(\mathbb R)$ is called 
$\mathbb C$-irreducible if it has no proper invariant subspace 
 in $\mathbb C^n$. 
By \cite{Goldman}
the character of a $\mathbb C$-irreducible representation 
has a smooth (non-singular) neighborhood in the variety of characters 
$X( \mathcal O^2, \mathrm{SL}_{n}(\mathbb R))
$.
In this paper we show that its composition with the
standard inclusion in $\mathrm{SL}_{n+1}(\mathbb R)$ 
has a character topologically singular in
$
X( \mathcal O^2, \mathrm{SL}_{n+1}(\mathbb R))
$, and we describe the singularity.

For a  
 $\mathbb C$-irreducible 
representation $\rho\colon\pi_1(\mathcal O^2)\to \mathrm{SL}_{n}(\mathbb R)$, let $ \mathbb R^n_\rho$ denote 
$ \mathbb R^n$ as 
 $\pi_1(\mathcal O^2)$-module via $\rho$ and  
  set  $d=\dim(H^1(\mathcal O^2,\mathbb R^n_\rho))$.
View also $\rho$ as a representation in  $ \mathrm{SL}_{n+1}(\mathbb R)$, by composing it with the standard inclusion 
$ \mathrm{SL}_{n}(\mathbb R)\subset \mathrm{SL}_{n+1}(\mathbb R)$ and so that it is no longer 
$\mathbb C$-irreducible.

\begin{Theorem}
 \label{Theorem:orientable}
Let $\mathcal O^2$ be a compact and orientable 2-orbifold, satisfying $\chi(\mathcal O^2)<0$. 
Let  $\chi_\rho\in X( \mathcal O^2, \mathrm{SL}_{n}(\mathbb R) )$  
be the character of a $\mathbb C$-irreducible representation  $ \rho$,
% $\rho\colon\pi_1(\mathcal O^2)\to \mathrm{SL}_{n}(\mathbb R)$, 
$d=  \dim(H^1(\mathcal O^2,\mathbb R^n_\rho))$, and $b= \dim(H^1(\mathcal O^2,\mathbb R))$ the 
first Betti number. 
A neighborhood of  $\chi_\rho$ in  
 $ X(\mathcal O^2, \mathrm{SL}_{n+1}(\mathbb R))$ is
 homeomorphic to
 $$
  \mathbb{R}^p\times \mathbb R^b\times  \mathrm{Cone}(X),
 $$
 where $X$ is as in Table~\ref{Table:Table1}
% % %  $$
% % %  X=\begin{cases}
% % %          \mathrm{UT}( S^{d-1}) & \textrm{if }\mathcal O^2 \textrm{ is closed and }n+1\textrm{ is even},\\
% % %             \mathrm{UT}( \mathbb{RP}^{d-1}) & \textrm{if }\mathcal O^2 \textrm{ is closed and }n+1\textrm{ is odd},\\
% % %           S^{d-1}\times S^{d-1} & \textrm{if }\mathcal O^2 \textrm{ has boundary and }n+1\textrm{ is even},
% % %          \\
% % %          ( S^{d-1}\times S^{d-1})/\!\sim\! & \textrm{if }\mathcal O^2 \textrm{ has boundary and }n+1\textrm{ is odd},
% % %    \end{cases}
% % %  $$
%  $ (\mathbb R^{p}\times \mathbb R^b\times X, \mathbb R^{p}\times\{0\}\times\{0\}) $, where
%  $$
%  \begin{cases}
%         U\times \mathbb R^b\times  \mathrm{Cone}( \mathrm{UT}( S^{d-1})) & \textrm{if }\mathcal O^2 \textrm{ is closed and }n+1\textrm{ is even},\\
%          U\times \mathbb R^b\times
%          \mathrm{Cone}( S^{d-1}\times S^{d-1}) & \textrm{if }\mathcal O^2 \textrm{ has boundary and }n+1\textrm{ is even},
%          \\
%             U\times \mathbb R^b\times  \mathrm{Cone}( \mathrm{UT}( \mathbb{RP}^{d-1})) & \textrm{if }\mathcal O^2 \textrm{ is closed and }n+1\textrm{ is odd},\\
%          U\times \mathbb R^b\times
%          \mathrm{Cone}( S^{d-1}\times S^{d-1})/\!\sim\! & \textrm{if }\mathcal O^2 \textrm{ has boundary and }n+1\textrm{ is odd},
%         \end{cases}
%  $$
%  where
and
$\mathbb R^p\times \{0\}\times\{0\}$ corresponds to
a smooth neighborhood in the variety of characters
 $X(\mathcal O^2, \mathrm{SL}_{n}(\mathbb R) ) $.
 \end{Theorem}

 \begin{table}[h!]
\centering \renewcommand{\arraystretch}{1.5}
 \begin{tabular}{|c|cc|}
 \hline $X$ & $n+1$ even & $n+1$ odd
  \\ \hline   $\mathcal O^2$ is closed &  $\mathrm{UT}( S^{d-1})$  &
  $\mathrm{UT}( \mathbb{RP}^{d-1})$ \\
%   \hline
   $\mathcal O^2$  has boundary &  $S^{d-1}\times S^{d-1}$ &
  $  ( S^{d-1}\times S^{d-1})/\!\sim\!$
   \\
    \hline
 \end{tabular}
\caption{The link $X$ of the singularity in Theorem~\ref{Theorem:orientable}.}
\label{Table:Table1}
 \end{table}

 Here $S^{d-1}\subset \mathbb R^d$ denotes the $(d-1)$-dimensional unit sphere; 
 $ \mathrm{UT} (S^{d-1})$, its unit tangent bundle
 $$
 \mathrm{UT} (S^{d-1})=\{(u, v)\in S^{d-1}\times S^{d-1}\subset \mathbb R^{2d}\mid u\cdot v=0 \};
 $$  
$  \mathrm{UT}( \mathbb{RP}^{d-1})$, the unit tangent bundle
of projective space; and $\sim$,  the involution of $S^{d-1}\times S^{d-1}$
acting as the antipodal   on each factor.
 The cone on a topological space $X$ is denoted by
$\mathrm{Cone}(X)$.
% \begin{Remark}
When $d=0$, we use the convention that 
$S^{-1}=\emptyset$ and 
$ \mathrm{Cone}( \emptyset)$ is a point.
% \end{Remark}
%  
 The factor $\mathbb R^b$ is realized by representations in the stabilizer of $\rho$ 
 in $\mathrm{SL}_{n+1}(\mathbb R) $, namely
 $\mathrm{diag}( \theta,\ldots,\theta, \theta^{-n})$ for a
 homomorphism $\theta\colon\Gamma\to \mathbb R^*$.
 The difference according to the parity of $n+1$
 may be explained by the fact that we take into account the action of
 stabilizer of $\rho$:
when $n+1$ is odd, the matrix
$\mathrm{diag}( -1,\ldots,-1, 1)$ stabilizes $\rho$ but it is
not central, though when $n+1$ is even,
$\mathrm{diag}( -1,\ldots,-1)$
  is central.

 \medskip

As motivation, when $n=3$ we consider projective structures on $\mathcal O^2$ modeled in $\mathbb P^2$,
and we view them as projective structures on $\mathcal O^2\times\mathbb R$ modeled in $\mathbb P^3$.
When $\mathcal O^2$ is closed and orientable,
the space of convex projective structures on $\mathcal O^2$ is a 
component of 
$X( \mathcal O^2, \mathrm{SL}_3(\mathbb R))$, homeomorphic to a 
cell, by Choi-Goldman \cite{ChoiGoldman}. When embedding $\mathrm{SL}_3(\mathbb R)$ in
$\mathrm{SL}_4(\mathbb R)$ by  Theorem~\ref{Theorem:orientable} we get a singularity
(with few exceptions for small orbifolds). 
More precisely (Corollary~\ref{Corollary:projective})
 a neighborhood of the character of its projective holonomy  in
 $ X(\mathcal O^2, \mathrm{SL}_4(\mathbb R))$ is
 homeomorphic to
 $$
 \begin{cases}
        \mathbb R^{p+b}\times  \mathrm{Cone}( \mathrm{UT}( S^{t-1})) & \textrm{if }\mathcal O^2 \textrm{ is closed}\\
         \mathbb R^{p+b}\times
         \mathrm{Cone}( S^{t-1}\times S^{t-1}) & \textrm{if }\mathcal O^2 \textrm{ has boundary}
        \end{cases}
 $$ where $t$ is the dimension of its Teichm\"uller space  and
 $p$ is the dimension of its Choi-Goldman's space of convex projective structures.
 In particular (Corollary~\ref{Corollary:turnover}) for a turnover $t=b=0$, so every deformation in
  $\mathrm{SL}_4(\mathbb R)$ of its projective holonomy is
  conjugate to a deformation in  $\mathrm{SL}_3(\mathbb R)$, as the cone of the empty set is a point.

% %   This inclusion of spaces of projective structures
% %   is intended to be an analog of the embedding of the Teichm\"uller space into the quasi-Fuchsian space.
% %   Viewed as varieties of characters, it is induced by  the inclusion of isometry groups
% %   $\mathrm{Isom}^+(\mathbb{H}^2)\subset\mathrm{Isom}^+(\mathbb{H}^3)$, namely
% %   $\mathrm{PSL}_2(\mathbb R)\subset \mathrm{PSL}_2(\mathbb C)$. By applying a Theorem of Goldman \cite{Goldman}
% %   (see Proposition~\ref{Proposition:Goldman} below) the corresponding
% %   varieties of characters are both smooth.

% \medskip

The theorem applies also to the Barbot component \cite{Barbot}.
Given a discrete and faithful representation
  of the fundamental group of a surface $F_g$ of genus $g\geq 2$ in $
  \mathrm{SL}_2(\mathbb R)$, Barbot proves in \cite{Barbot} that
  the composition
  with the standard embedding in $\mathrm{SL}_3(\mathbb R)$ yields
    Anosov representations that do not lie in the  Hitchin
  component. The
  corresponding component in the character variety $X(F_g,\mathrm{SL}_3(\mathbb R))$ is singular by Theorem~\ref{Theorem:orientable}, and a neighborhood of singular points is homeomorphic to
  $
  \mathbb R^{8g-6}\times  \mathrm{Cone}( \mathrm{UT}( \mathbb{RP}^{4g-5}))
  $.

\medskip

At the end of the  paper we also discuss
the deformation space of projective structures on a hyperbolic 3-orbifold $\mathcal O^3$ of finite type.
For this purpose, the holonomy representation $\rho\colon\pi_1(\mathcal O^3)\to \mathrm{SO}(3,1)$
of the hyperbolic structure
is composed
with the standard inclusion $\mathrm{SO}(3,1)\to\mathrm{SL}_4(\mathbb R)$, yielding the
holonomy of the corresponding projective structure.
Thus a neighborhood of the character of $\rho$ in
the variety of characters
$X(\mathcal O^3, \mathrm{SL}_4(\mathbb R))$ yields the deformation space of this projective structure.

To put our result in context, we recall a theorem of
Ballas, Danciger, and Lee \cite[Theorem~3.2]{BDL}. According to their theorem, if a finite volume orbifold
$\mathcal O^3$ is infinitesimally rigid with respect to the boundary, then the character of its hyperbolic
holonomy is a smooth point of $X(\mathcal O^3, \mathrm{SL}_4(\mathbb R))$. Here
``infinitesimally rigid with respect to the boundary'' is a technical hypothesis on
cohomology on the Lie algebra (Definition~\ref{Definition:infinitesimal}),  it essentially means
that all infinitesimal deformations
of $\mathcal O^3$ come from its ends. We  prove the following theorem:

\begin{Theorem}
\label{Theorem:3dimIntro}
Let $\mathcal O^3$ be a compact orientable orbifold, with $\partial \mathcal O^3\neq \emptyset$ and hyperbolic interior
$\mathrm{int}( \mathcal O^3 )$ so  that it is not elementary, nor Fuchsian.
Assume that $\mathcal O^3$ is infinitesimally projectively rigid
with respect to the boundary.
Then
 the character of the
hyperbolic holonomy is a smooth point of
$X(\mathcal O^3, \mathrm{SL}_4(\mathbb R))$
if,  and only if,
all ends of
$\mathrm{int}( \mathcal O^3 )$ are either non-Fuchsian or turnovers. Furthermore, the singularity is quadratic.
 \end{Theorem}

We say that one end is \emph{Fuchsian} if its holonomy is a Fuchsian group: namely the end corresponding to a totally geodesic
boundary component.

% In this paper we assume the same technical condition, but instead of finite volume the hyperbolic
% orbifold  $\mathcal O^3$ is only required
% to be of finite type. In
% Theorem~\ref{Theorem:3dim} we show that $X(\mathcal O^3, \mathrm{SL}_4(\mathbb R))$ is singular precisely
% when some of the ends of $\mathcal O^3$ is Fuchsian other than a turnover.

 \medskip

Let us comment on the proof of Theorem~\ref{Theorem:orientable}.
Under the hypothesis of this theorem,
when $\partial\mathcal O^2\neq\emptyset $
the variety of representations (not of characters)
$\hom(\pi_1(\mathcal O^2), \mathrm{SL}(n+ 1,\mathbb R)) $ is smooth at~$\rho$.
% % , the composition
% % of an irreducible representation in $ \mathrm{SL}_{n}(\mathbb R) $ with the standard embedding
% % in $\mathrm{SL}(n+ 1,\mathbb R)$. 
The singularity occurs
in the variety of characters, when considering the real GIT quotient.
This singularity appears because $\rho$ is not irreducible, and
there is a one-parameter subgroup of $\mathrm{SL}_{n+1}(\mathbb R)$
that commutes with $\mathrm{SL}_{n}(\mathbb R)$. 
% This may be compared with the embedding of Teichm\"uller space in quasi-Fuchsian space, both smooth, as the commutator
% of $\mathrm{Isom}^+(\mathbb H^2)$ in $\mathrm{Isom}^+(\mathbb H^3)$ is trivial.
% 
To understand this singularity  when $\partial\mathcal O^2\neq\emptyset $, notice that
$\mathbb R^n_\rho$ and $\mathbb R^n_{\rho^*}$  appear in the complement of
$\mathfrak{sl}(n,\mathbb R)$ in $\mathfrak{sl}(n+1,\mathbb R)$,
where 
$\rho^*(\gamma)=\rho(\gamma^{-1})^t$ is the contragredient representation. 
When $n+1$ is even the action of the commutator of $\mathrm{SL}_{n}(\mathbb R)$ on
$H^1(\mathcal O^2,\mathbb R^n_{\rho^*})\oplus H^1(\mathcal O^2,\mathbb R^n_\rho)$
is equivalent to the action of 
$\mathbb R-\{0\}$ on  $\mathbb R^d\times\mathbb R^d$
defined by 
$t\cdot (x,y)\mapsto (t^{n+1} x, t^{-n-1} y)$, for $t\in \mathbb R-\{0\}$ 
and $x, y\in \mathbb R^d$.
The quotient   $(\mathbb R^d\times\mathbb R^d)/\!/ \mathbb R_{>0} $ is homeomorphic to 
the cone $|x|^2=|y|^2$, 
$(x,y)\in \mathbb R^d\times\mathbb R^d$. When $n+1$ is odd, in addition  we 
take into account of the antipodal
on
$\mathbb R^n_{\rho^*}\oplus \mathbb R^n_\rho$.

When $ \mathcal O^2  $ is closed,  the variety of representations of 
$\hom(\pi_1(\mathcal O^2), \mathrm{SL}(n+ 1,\mathbb R)) $ has a quadratic singularity, by 
a theorem of Goldman \cite{GoldmanMaryland}
(see also Goldman and Millson  \cite{GoldmanMillson}  and Simpson \cite{Simpson}). Then both singularities
have to be combined, the quadratic singularity and the singularity from passing to the quotient.
Thus for $x, y\in \mathbb R^d$, we combine the equality $|x|^2=|y|^2$ with $x\cdot y=0$,
that yields the cone in the unit tangent bundle of the sphere  when $n+1$
is even, or its
quotient by the antipodal when $n+1$ is odd.

The paper also discusses non-orientable orbifolds, 
so we consider  the group
$$\mathrm{SL}^{\pm}_n(\mathbb R ) 
 =\{A\in\textrm{GL}_n(\mathbb R)\mid \det (A)=\pm 1\}.
 $$
In this case there are two natural extensions  of 
$ \mathrm{SL}^{\pm}_n(\mathbb R )$.  One extension is in  $\mathrm{SL}_{n+1}(\mathbb R)$, so that
every transformation of $\mathbb P^{n-1}$  extends to an
orientation preserving transformation of $\mathbb P^n$.
Namely, every matrix $A\in \mathrm{SL}^{\pm}_n(\mathbb R )$ is mapped to a matrix
$ \mathrm{SL}_{n+1}(\mathbb R )$ by adding a last raw and last column of zeros, except for the
$(n+1,n+1)$ entry that equals $\det(A)$.
The other extension   $\mathrm{SL}^{\pm}_{n}(\mathbb R)\to \mathrm{SL}^{\pm}_{n+1}(\mathbb R)$
is just the trivial extension, so that the $(n+1,n+1)$ entry equals $1$.
Theorem~\ref{Theorem:orientable} is easily adapted to the non-orientable case, just 
taking into account these different extensions, 
see~Theorem~\ref{Theorem:nonorientable} for details.

In this paper we describe the topological singularities.
The variety of characters is homeomorphic to a real semi-algebraic set 
(cf.~Theorem~\ref{Thm:semialgebraic}), in particular every character has a neighborhood homeomorphic to a semi-analytic set. We shall not discuss this semi-analytic structure.

\medskip

\paragraph{Organization of the paper}
Section~\ref{Section:prelim} is devoted to preliminaries on
varieties of (real) representations and characters, as well as cohomology and orbifolds. In particular we review Goldman's results.
In Section~\ref{Section:DefSpace} we prove the main theorem, 
as well as its generalization to non-orientable orbifolds.
Section~\ref{Section:proj2} applies the main theorem to deformation spaces of
projective 2-orbifolds viewed in the space of three-dimensional projective
structures, including some examples.
 Finally we apply these results in
Section~\ref{Section:projdef3} to determine when the space of projective structures on hyperbolic 3-orbifolds is singular (assuming that
all infinitesimal projective deformations come from the ends of the hyperbolic three-orbifold).
% %
% % \medskip
% %
 \paragraph{Acknowledgements} I am indebted to  useful suggestions
of the referee, as well as the hospitality of the
% % MPI-MiS in Leipzig.
 Max Plank Institute for Mathematics in the Sciences in Leipzig.

\section{Preliminaries on varieties of characters}
\label{Section:prelim}

This first section is devoted to preliminaries on varieties of (real) representations and characters.

\subsection{Variety of characters for real reductive groups}
\label{Section:real characters}
 
We recall the basic results on varieties of representations and characters in 
algebraic real reductive groups. 
Let $\Gamma$ be a finitely generated discrete group (eg the fundamental group of a compact 
two-orbifold) and  
let $G$ be either
$$
G=\mathrm{SL}_n(\mathbb R ) \qquad \textrm{ or }
\qquad
G=\mathrm{SL}^{\pm}_n(\mathbb R )
% $, where 
% $
%  \textrm{SL}_{\pm }(n, \mathbb R)
 =\{A\in\textrm{GL}_n(\mathbb R)\mid \det (A)=\pm 1\},
$$
for $n\geq 2$.
The results here apply to more general
semi-simple real algebraic groups (and some definitions need to be adapted), but we consider only those groups.

The \emph{variety of representations} $\hom( \Gamma, G )$ is a real algebraic set.
The group $G$ acts by conjugation 
on  $\hom( \Gamma, G )$
and to define the  real GIT quotient we follow the results
of   Luna \cite{Luna}, Richardson \cite{RichardsonDuke},
and 
Richardson and Slodowy \cite{RichardsonSlodowy}.
See also a modern treatment by 
B\"ohm and Lafuente \cite{BohmLafuente}, as well as 
an interesting approach using symmetric spaces by Parreau
\cite{Parreau}. 

% % We start with results of   Lun

\begin{Theorem}[\cite{Luna,  RichardsonSlodowy}]
\label{Thm:semialgebraic}
 The closure of each orbit by conjugation 
 of $G$ on $\hom( \Gamma, G )$ 
 has precisely one closed orbit. 
 Furthermore, the space  of closed orbits is homeomorphic to a real semi-algebraic set.
\end{Theorem}

Here the space of closed orbits is equipped with the quotient topology of
a subset of $\hom( \Gamma, G )$.
One of the ingredients of Theorem~\ref{Thm:semialgebraic}
is Kempf-Ness theorem for real coefficients, proved first in \cite{RichardsonSlodowy}
(see also \cite{BohmLafuente}). Kempf-Ness
 provides an algebraic subset $\mathcal M\subset \hom( \Gamma, G )$ that intersects all closed orbits
and only closed orbits. Furthermore the space of closed orbits is homeomorphic to $\mathcal M/K$
for $K\subset G$ compact (and therefore it is homeomorphic to a semi-algebraic
set).

Using  Theorem~\ref{Thm:semialgebraic} one defines
 the \emph{variety of characters} $X(\Gamma, G)$ as the space of closed orbits.
 This construction
 is also called the \emph{real or 
$\mathbb R-\mathrm{GIT}$ quotient}:
 $$
X(\Gamma, G)=
\hom( \Gamma, G )/\!/ G.
$$
It is the ``Hausdorff quotient'' of the  topological quotient
$ \hom( \Gamma, G )/ G $:
every continuous map from
$ \hom( \Gamma, G )/ G $
to a Hausdorff space $ Y$
factors through $
\hom( \Gamma, G )/\!/ G
$ in a unique way.
% % $$
% % \begin{tikzcd}
% %   \hom( \Gamma, G )/G \arrow[d] \arrow[r] & Y \\
% % \hom( \Gamma, G )/\!/ G \arrow[ur, dashed]   &
% % \end{tikzcd}
% % $$

 Next we also mention a theorem that identifies the orbits that are closed.

\begin{Definition}
 A representation $\rho\in 
\hom( \Gamma, G )$ is called:
\begin{itemize}
 \item $\mathbb R$-\emph{irreducible}
or $\mathbb R$-\emph{simple}
if it has 
no proper invariant  subspace in $\mathbb R^n$;
 \item $\mathbb C$-\emph{irreducible}
or $\mathbb C$-\emph{simple}
if it has 
no proper invariant  subspace in $\mathbb C^n$;
 \item  \emph{semi-simple}
if $\mathbb R^n=E_1\oplus\cdots\oplus E_k$ so that each $E_i$ is $\rho$-invariant and the restriction of $\rho$ to $E_i$ is simple.
\end{itemize}
\end{Definition}

Notice that $\mathbb C$-irreducibility implies 
$\mathbb R$-irreducibility. However, the converse is not true,
consider
 representations in $\mathrm{SO}(2)$.
 Notice also that semi-simplicity does not depend on considering
 $\mathbb R^n$ or  $\mathbb C^n$ (an $\mathbb R$-irreducible space 
 decomposes
 as a possibly trivial direct sum of $\mathbb C$-irreducible spaces).
 The following 
 lemma is elementary:

\begin{Lemma}
\label{Lemma:semisimple}
 A representation $\rho\in\hom(\Gamma, G)$ is semi-simple
 iff every invariant subspace $V\subset\mathbb R^n$ has an invariant complement,
 eg a $\rho$-invariant subspace $W\subset\mathbb R^n$ satisfying 
 $V\oplus W= \mathbb R^n$.
\end{Lemma}

\begin{Theorem}[\cite{RichardsonDuke}]
Let $\rho\in 
\hom( \Gamma, G )$. The orbit  $G\cdot\rho$  is closed if and only if $\rho$ is semi-simple. 
\end{Theorem}

From this theorem, we can think of elements in $X(\Gamma, G)$
as conjugacy classes of semi-simple representations, that we may call
characters by abuse of notation. Furthermore we may talk about irreducible or simple characters.

The center of $G$ is $\mathcal Z(G)=\{h\in G\mid hg=gh\textrm{ for every }g\in G \}$.

\begin{Definition}
An \emph{analytic slice} at $\rho$
 is  analytic subvariety
 $S\subset   \hom(\Gamma, G)$ such that
  \begin{enumerate}[(i)]
  \item   $S\cap G\cdot\rho=\{\rho\}$,
%   $G_\rho(S)=S$, and
%  $G_\rho=\{g\in G\mid g\cdot S\cap S\neq \emptyset\}$
 \item There is an bi-analytic map
 $G/\mathcal Z(G)\times S\cong G\cdot S$ and $ G\cdot S$ is a  neighborhood of $G\cdot \rho$.
%  \item The map $(g, s)\mapsto g\cdot m$ induces $G$-equivariant retraction
%  $G\cdot S\to G\cdot\rho$.
 \end{enumerate}
\end{Definition}

The following is \cite[Theorem~1.2]{JohnsonMillson} by Johnson and Millson:

\begin{Theorem}[\cite{JohnsonMillson}]
\label{Theorem:slices}
Let $\rho\in \hom(\Gamma, G)$ be a $\mathbb C$-irreducible representation. Then 
the action by conjugation admits an analytic slice at $\rho$.
\end{Theorem}

% Here $\mathcal Z(G)$ denotes the center of $G$ (and acts trivially by conjugation).
We discuss a more general version of this theorem in Theorem~\ref{Theorem:Slice}.
In general the slice takes into account the stabilizer of $\rho$, when the representation is irreducible, the stabilizer is just the center of $G$:

\begin{Lemma}
\label{lemma:stabirr}
Let $\rho\in \hom(\Gamma, G)$ be a $\mathbb C$-irreducible representation. Then the stabilizer by conjugation of $\rho$ is the center of $G$:
$
\mathrm{Stab}_G(\rho)=\mathcal Z(G)
$.
\end{Lemma}

\begin{proof}
 The center is always contained in the stabilizer and we prove the other inclusion. Let $g\in G$ be an element of the stabilizer of $\rho$. 
%  If 
%  $G=\mathrm{PSL}(n,\mathbb R ) $ or  $G=\mathrm{PGL}(n,\mathbb R ) $, we may take a represntative $g\in \mathrm{GL}(n,\mathbb R )$. 
 The relation $g\rho=\rho g$ implies that $\rho$ preserves the  eigenspaces of $g$, thus $g$ has no proper   eigenspaces. Equivalently, $g$ is an scalar multiple of the identity, so an element of the
 center.
\end{proof}

% 
% Following the definition of smooth slice in 
% \cite[IX.5.3]{BorelWallach}:
%  
% \begin{Definition} 
%  An analytic slice $S$ at $\rho$ is an analytic subvariety 
%  $S\subset   \hom(\Gamma, G)$ such that
%  \begin{enumerate}[(i)]
%   \item   $S\cap G\cdot\rho=\{\rho\}$, $G_\rho(S)=S$, and 
%  $G_\rho=\{g\in G\mid g\cdot S\cap S\neq \emptyset\}$
%  \item There is an bi-analytic map $G\times_{G_\rho}S\cong G\cdot S$ with an open neighborhood of $G\cdot m$ in $M$.
%  \item The map $(g, s)\mapsto g\cdot s$ induces $G$-equivariant retraction
%  $G\cdot S\to G\cdot\rho$.
%  \end{enumerate}
% \end{Definition}

\begin{Remark}
 It follows from Theorem~\ref{Theorem:slices}
that there exists a well
defined analytic structure in a neighborhood of a $\mathbb C$-irreducible character, as 
in~\cite{JohnsonMillson}. Without assuming irreducibility,
one should consider semi-analytic structures on $X(\Gamma, G)$.
 \end{Remark}

% % 
% % ***
% % 
% % Explain def of Analytic slice at $\rho$
% % in BorelWallach
% %  Borel-Wallach[5, IX.5.3] :
% %  
% % $S\subset   \hom(\Gamma, G)$ analytic subvariety such that:
% % \begin{itemize}
% %  \item $S\cap G\cdot\rho=\{\rho\}$, $G_\rho(S)=S$, and 
% %  $G_\rho=\{g\in G\mid g\cdot S\cap S\neq \emptyset\}$
% %  \item $G\times_{G_\rho}S\cong G\cdot S$ open neighborhood of $G\cdot m$ in $M$.
% %  \item $(g, s)\mapsto g\cdot m$ induces $G$-equivariant retraction
% %  $G\cdot S\to G\cdot m$ (stabilizers in $G_m$)
% %  
% %  
% %  
% % \end{itemize}
% % 
% % 
% %  
% %  
% %  
% % **
% % **
% % **

% To discuss the local structure we need to use tools from cohomology and the bar resolution.

\subsection{Group cohomology}
\label{Section:cohomology}

% This section reviews classical tools of group cohomology, it can be skipped and used as refrence for further use. 
We follow \cite{Brown} for basics on group cohomology.
% and \cite{PortiDim} for orbifolds.
% 
Fix a finitely generated group $\Gamma$, a $\mathbb R$-vector space $V$, and a representation
$\Gamma\to\mathrm{GL}(V)$, so that $V$ is called a \emph{$\Gamma$-module}.
We are mainly interested in the case where there is a representation 
$\rho\colon\Gamma\to G$ and $V=\mathfrak g$, that is a $\Gamma$-module via the adjoint of $\rho$,
but we shall also consider other $\Gamma$-modules. 

The $i$-chains of the  \emph{bar resolution} are defined as maps from 
$\Gamma\times\overset{(i)}\cdots\times \Gamma$ to $V$:
$$
C^i(\Gamma,V)= \{\theta\colon \Gamma\times\overset{(i)}\cdots\times \Gamma\to V\},
\ \textrm{ for }i>0
\qquad\textrm{ and }\qquad C^0(\Gamma,V)=V.
$$
% with the convention C^0(\Gamma,V)=V.
The coboundary $\delta^i\colon C^i(\Gamma,V)\to C^{i+1}(\Gamma,V)$ is defined by
\begin{multline*}
\delta^i(\theta)(\gamma_0,\ldots,\gamma_{i})
% \\
=\gamma_0\theta(\gamma_1,\ldots,\gamma_i)
\\
+
\sum_{j=0}^{i-1}(-1)^{j+1} \theta(\gamma_0,\ldots, \gamma_j\gamma_{j+1},\ldots, \gamma_i)
\\
+(-1)^{i+1}\theta(\gamma_0,\ldots, \gamma_{i-1}) . 
\end{multline*}
The space of cycles and coboundaries are denoted respectively by
$$
Z^i(\Gamma, V)=\ker \delta^i  \qquad \textrm{and}\qquad
B^i(\Gamma, V)=\operatorname{Im} \delta^{i-1} 
$$
so that the cohomology is $$
H^i(\Gamma, V)= Z^i(\Gamma, V)/ B^i(\Gamma, V).
$$

The zero-th cohomology group is naturally isomorphic to the subspace of invariants 
\begin{equation}
 \label{eqn:H0}
H^0(\Gamma, V)\cong Z^0(\Gamma, V)\cong V^{\Gamma},
\end{equation}
because $\delta^0(v)(\gamma)= \gamma v-v$ for $\gamma\in\Gamma$ and $v\in V$.

\subsection{Products in cohomology}

 Let $V_1$, $V_2$ and $V_3$ be $\Gamma$-modules and 
let 
$$
\varphi\colon V_1\times V_2\to V_3
$$ be a bilinear map that is $\Gamma$-equivariant. 
Combined with the cup product it induces a pairing
in cohomology, that we denote by $\varphi(\cdot\cup\cdot)$.
$$ \varphi(\cdot\cup\cdot)\colon
 H^i(\Gamma, V_1)\times 
 H^{j}(\Gamma,  V_2)\overset\cup\to 
 H^{i+j}(\Gamma,  V_1\otimes V_2 )\overset\varphi\to 
 H^{i+j}(\Gamma, V_3).
 $$

We are  interested in the explicit description for  1-cocycles. 
The space of $1$-cocycles is 
$$
Z^1(\Gamma, V)= \{\theta\colon \Gamma\to V\mid \theta(\gamma_1\gamma_2)=\theta(\gamma_1)+\gamma_1\theta(\gamma_2), 
\forall \gamma_1,\gamma_2\in \Gamma\}.
$$
The space of 1-coboundaries is
$$
B^1(\Gamma, V)= \{\theta_v \in  Z^1(\Gamma, V)\mid \textrm{ there is a }v\in V\textrm{ s.t. } 
\theta_v(\gamma)= \gamma v - v,\forall\gamma\in\Gamma\}  .
$$
For 1-cocycles, we describe the cup product 
following \cite{Brown}:
if $z_i\in Z^1(\Gamma, V_i)$, then $
\varphi(z_1\cup z_2)\in Z^2(\Gamma, V_3)
$ 
is the 2-cocycle
\begin{equation}
\label{eqn:cup}
\varphi(z_1\cup z_2)(\alpha,\beta)= \varphi(z_1(\alpha), \alpha z_2(\beta)),\qquad
\forall \alpha,\beta\in\Gamma.
\end{equation}
In addition we have:

\begin{Lemma} 
\label{lemma:symmetry}
Assume $V_2=V_1$, 
then $$
\varphi(\cdot\cup \cdot)\colon H^1(\Gamma, V_1)\times H^1(\Gamma, V_1)\to H^2(\Gamma, V_3)
$$ 
is skew-symmetric when $\varphi$ is symmetric, and 
symmetric when  $\varphi$ is skew-symmetric. 
\end{Lemma}

\begin{proof}
Consider the 1-chain $c\colon\Gamma\to V_3$ defined by $c(\gamma)=\varphi(z_1(\gamma),  z_2(\gamma))$. 
A direct application of the definition of the coboundary $\delta^1$
yields 
\begin{align*}
\delta^1(c)(\alpha, \beta)& = \varphi(z_1(\beta),  z_2(\beta)) - \varphi(z_1(\alpha\beta ),   z_2(\alpha\beta))
+ \varphi(z_1(\alpha),  z_2(\alpha)) \\
&= -\varphi(z_1(\alpha), \alpha z_2(\beta))
-\varphi(\alpha z_1(\beta),  z_2(\alpha)), \qquad \qquad \forall \alpha, \beta\in \Gamma .
\end{align*}
Hence by \eqref{eqn:cup} we deduce: 
$$-\delta^1(c)=  \varphi(z_1\cup z_2)\pm  \varphi(z_2\cup z_1) $$
where the sign in $\pm$ is $+$ for $\varphi$ symmetric, and $-$ for 
$\varphi$ skew-symmetric.
 \end{proof}

Below we describe the role of 1-cocycles in tangent spaces and infinitesimal deformations.

\subsection{Tangent space to the variety of representations}

Let $\rho\colon\Gamma\to G$ be a representation. We view $\mathfrak g$ as a $\Gamma$-module by composing $\rho$ with the
the adjoint: $\mathrm{Ad}\colon G\to \mathrm{End}(\mathfrak g) $.

We describe Weil's construction \cite{LubotzkyMagid, Weil}. Given a 1-cocycle  $d\in Z^1(\Gamma,\mathfrak g)$,
the assignment for each $\gamma\in \Gamma$:
$$
\gamma\to (1+ t \, d(\gamma)) \rho(\gamma)
$$
is an infinitesimal path in $\hom(\Gamma, G)$; namely a path of representations
up to a factor $t^2$, or a 1-jet  from $(\mathbb{R}, 0)$ to 
$( \hom(\Gamma, G), \rho) $, eg an element in $J^1_{0,\rho}(\mathbb{R}, \hom(\Gamma, G))$. 

\begin{Proposition}\cite{LubotzkyMagid, Weil}
 Weil's construction yields an isomorphism
 $$
 Z^1(\Gamma,\mathfrak g)\cong  T^{\mathrm{Zar}}_\rho
\hom(\Gamma, G)
 $$
that maps $ B^1(\Gamma,\mathfrak g)$ to the tangent space 
to the orbit by conjugation. 
\end{Proposition}
  
Here $ T^{\mathrm{Zar}}_\rho
\hom(\Gamma, G)
 $ means the Zariski tangent space as scheme, as the 
 coordinate ring may be non-reduced, cf.~\cite{HeusenerPorti23}.
Using Theorem~\ref{Theorem:slices} on the existence of a slice, we get:

\begin{Theorem}
Let $[\rho]\in X(\Gamma, G)$ be a $\mathbb C$-irreducible character. Weil's construction factors to an isomorphism:
$$
 H^1(\Gamma,\mathfrak g)\cong  T^{\mathrm{Zar}}_{[\rho]}
X(\Gamma, G).
$$
\end{Theorem}

We discuss the general case (without assuming irreducibility) in Section~\ref{Section:slice}.
% % 
% % *** 
% % 
% % semisimple?
% % 
% % 
% % ***
We also need  a result in the zero-th cohomology group in the 
$\mathbb C$-irreducible case.

\begin{Lemma}
\label{Lemma:H0trivial}
 If $\rho $ is $\mathbb C$-irreducible, then the space of invariants of the Lie algebra is trivial:
 $\mathfrak g^{\mathrm{Ad} {\rho(\Gamma)}}=0$.
 In particular $H^0(\Gamma, \mathfrak g)= 0$.
\end{Lemma}

\begin{proof}
 Assume $H^0(\Gamma, \mathfrak g)\neq 0$,   by~\eqref{eqn:H0}
 there exists $0\neq a\in  \mathfrak g^{\mathrm{Ad}{\rho(\Gamma)}}$. Then
 the one-parameter group $\{\exp ( t\, a)\mid t\in\mathbb R\}$
 commutes with the image of $\rho$. By considering the eigenspaces of  $\exp (  a)$, this commutativity contradicts that $\rho$ is 
 $\mathbb C$-irreducible by
 Lemma~\ref{lemma:stabirr}.
\end{proof}

 \subsection{Orbifolds}

% % In the next section we shall use Weil's construction on the interpretation of 
% % $Z^1(\Gamma,\mathfrak g)$ as the Zariski tangent space, as well as the sequence of obstructions 
% % to integrability in  $H^2(\Gamma,\mathfrak g)$.

We relate the group cohomology with orbifold cohomology, that can be defined as the simplicial cohomology
for a CW-complex structure on the orbifold \cite[\S3]{PortiDim}. 
An orbifold is \emph{very good} if it has a finite orbifold covering that is 
a manifold, and in this case the orbifold cohomology is the equivariant cohomology of the manifold
covering.
We are also interested in orbifolds that are \emph{aspherical}:
the universal covering is a contractible manifold. 

In this paper we consider 2-dimensional orbifolds with negative Euler characteristic (eg hyperbolic) and
hyperbolic 3-orbifolds, hence very good and aspherical.

\begin{Lemma} If $\mathcal O^n$ is an aspherical and very good orbifold, then there is a natural isomorphism
$$
H^i(\mathcal O^n, V)\cong H^i(\pi_1(\mathcal O^n), V)
$$
\end{Lemma}

See for instance \cite[Prop.~3.4]{PortiDim} for a proof (for an aspherical manifold this is standard, and for
a very good orbifold use a manifold covering and work equivariantly).

We shall need  Poincar\'e duality for cohomology with coefficients. Let $V_1$ and $V_2$ be $\Gamma$-modules and 
let $\varphi\colon V_1\times V_2\to \mathbb R$ be a  pairing that is $\Gamma$-equivariant. 
Combined with the cup product it induces a pairing
in cohomology, that we denote by $\varphi(\cdot\cup\cdot)$.
\begin{multline*} \varphi(\cdot\cup\cdot)\colon
 H^i(\mathcal O^n, V_1)\times 
 H^{j}(\mathcal O^n, \partial \mathcal O^n; V_2)\overset\cup\to 
 H^{i+j}(\mathcal O^n, \partial \mathcal O^n; V_1\otimes V_2 )
 \\
 \overset\varphi\to 
 H^{i+j}(\mathcal O^n, \partial \mathcal O^n; \mathbb R).
\end{multline*}

\begin{Theorem}[Poincar\'e duality with coefficients]
\label{TheoremPD}
 Let $\mathcal O^n$ be a  compact, connected, orientable, and very good n-orbifold.
 For $V_1$, $V_2$ and $\varphi$ as above; if $\varphi$ is a perfect pairing, then the product
 $$ \varphi(\cdot\cup\cdot)\colon
 H^k(\mathcal O^n, V_1)\times 
 H^{n-k}(\mathcal O^n, \partial \mathcal O^n; V_2)\to 
 H^n(\mathcal O^n, \partial \mathcal O^n; \mathbb R)
 \cong\mathbb R
 $$
 is a perfect pairing.
\end{Theorem}

See for instance 
\cite[Proposition 3.4]{PortiDim}
for a proof of this orbifold version of Poincar\' e duality.

% We finish the section with an explicit construction of the cup product in dimension 1:

\subsection{Obstructions to integrability}
\label{Section:tangent}

Here we review the theory of obstructions to integrability. We start by reviewing the following theorem of Goldman:

\begin{Proposition}[\cite{Goldman}]
\label{Proposition:Goldman}
Set $\Gamma= \pi_1(\mathcal O^2)$ and
 assume $\rho\colon\Gamma\to G$ is $\mathbb C$-irreducible. Then:
 \begin{enumerate}[(i)]
  \item $\hom(\Gamma, G)$ is  smooth at $\rho$ and 
  $Z^1(\Gamma,\mathfrak g)$ is isomorphic to $T_\rho\hom(\Gamma, G)$.
  \item $X(\Gamma, G)$ is smooth at the character $[\rho]$ and 
  $H^1(\Gamma,\mathfrak g)$ is isomorphic to $T_{[\rho]} X(\Gamma, G)$.
 \end{enumerate}
 Furthermore, $X(\Gamma, G)$ is locally homeomorphic to the topological quotient  $\hom(\Gamma, G)/G$.
\end{Proposition}

The statement of this proposition deals with
$T$ instead of $T^{\mathrm{Zar}}$, because in the smooth case,
the Zariski tangent space equals the standard tangent space (here smooth means smooth 
not only as a variety but as scheme,
in particular reduced).

A key tool for
Proposition~\ref{Proposition:Goldman}
is Goldman's obstruction
to integrability, defined in \cite{Goldman}. 
The first obstruction  to integrability uses the cup product
combined with the Lie bracket, as in \eqref{eqn:cup}. At the level of 1-cocycles it is defined as follows:
for $\sigma,\varsigma\in  Z^1(\Gamma,\mathfrak{g})$, 
% a bilinear form  $[\sigma\cup\varsigma]\in 
% Z^2(\Gamma,\mathfrak{g})$ is defined by 
$$
\begin{array}{rcl}
[\sigma\cup\varsigma]\colon
\Gamma\times\Gamma & \to & \mathfrak g \\
(\gamma_1,\gamma_2) & \mapsto & [\sigma(\gamma_1), \mathrm{Ad}_{\rho(\gamma_1)}\varsigma(\gamma_2)] 
\end{array}
$$
As the Lie bracket is skew-symmetric, by Lemma~\ref{lemma:symmetry} the induced product in cohomology 
is a symmetric bilinear form 
% 
% and it induces a symmetric product, by compining the cup product $\cup$ with the Lie bracket $[,]$ 
$$
[\cdot\cup\cdot]\colon H^1(\Gamma, \mathfrak{g})
\times H^1(\Gamma, \mathfrak{g}) \to H^2(\Gamma, \mathfrak{g}) 
$$
For a tangent vector $z\in Z^1(\Gamma,\mathfrak g)$, the first obstruction to integrability of $z$
is the class  of $[z\cup z]$ in $H^2( \Gamma, \mathfrak{g})$, see
\cite{Goldman,HPS} for details.

The element $z\in Z^1(\Gamma,\mathfrak g)$ can be seen as a deformation of representations at first order,
a 1-jet $J^1_{0,\rho}(\mathbb R,\hom(\Gamma, G) )  $.
When $[z\cup z]\sim 0$   there is a deformation up to second order, equivalently it extends to  a 2-jet  in  $J^2_{0,\rho}(\mathbb R,\hom(\Gamma, G) )  $. We state the existence of Goldman's obstructions as a lemma (see  \cite{Goldman}
or \cite{HPS} for a proof):

\begin{Lemma} [Goldman \cite{Goldman}]
\label{Lemma:GoldmanObstructions}
 An $n$-jet in  $J^n_{0,\rho}(\mathbb R,\hom(\Gamma, G) )  $ has a natural obstruction in
$H^2(\Gamma, \mathfrak{g})$ that vanishes if and only if it extends to an
$(n+1)$-jet.
\end{Lemma}

% % 
% % 
% % There is a sequence of obstructions in
% % $H^2(\Gamma, \mathfrak{g})$ such that when the $n$-th 
% % obstruction vanishes, there exists a polynomial that is 
% % a representation up to order $n+1$, and define the $(n+1)$-th obstruction \cite{Goldman,HPS}. 
In the proof of 
Proposition~\ref{Proposition:Goldman}, Goldman uses that 
 $H^2(\Gamma, \mathfrak{g}) =0$,
 by Lemma~\ref{Lemma:H0trivial} and Poincar\'e duality, 
so all obstructions vanish.

When $H^2(\Gamma, \mathfrak{g})$ does not vanish, for 
certain classes of groups $\Gamma$ including $\pi_1(\mathcal O^2)$, Goldman \cite{GoldmanMaryland}, 
Goldman and Millson  \cite{GoldmanMillson},  and Simpson \cite{Simpson} prove that  the cup product is
the only obstruction to integrability, hence the singularities in 
$\hom(\pi_1(\mathcal O^2), G)$ are at most quadratic.
In terms of analytic geometry, the result is as follows:

\begin{Theorem}[\cite{GoldmanMaryland,GoldmanMillson,Simpson}]
\label{Theorem:Quadratic} Let $\mathcal O^2$ be closed and oriented. If the representation  $\rho\in \hom(\pi_1(\mathcal O^2)   , G)$ is semi-simple, then the analytic germ of 
 $\hom(\pi_1(\mathcal O^2), G)$ at $\rho$ is 
 analytically equivalent to the quadratic cone
 $$
 \{ z \in Z^1(\Gamma,G)\mid [ z\cup z] \sim 0 \}.
 $$
\end{Theorem}

\section{Deformation space}
\label{Section:DefSpace}

In this section we prove Theorem~\ref{Theorem:orientable}. It requires a slice theorem, in Subsection~\ref{Section:slice}, and some discussion
on the decomposition of the Lie algebra, Subsection~\ref{Section:Decomposition}. 
Theorem~\ref{Theorem:orientable}
is proved in Subsection~\ref{Section:ProofMT}.
In Subsection~\ref{subsection:nonorientable} we discuss the non-orientable case.

\subsection{An   analytic slice at semisimple representations}
\label{Section:slice}

There are several results in the literature dealing with slices of algebraic actions of reductive groups. 
For completeness, we have decided to include a slice theorem, that we did not find in the literature in this precise form.

We shall use some notation of real analytic geometry: analytic varieties are viewed
as subsets of some Euclidean space $\mathbb R^n$. Furthermore,  a \emph{germ at a point} means
an equivalence class of open sets centered at a point.
For convenience,  we use the term \emph{germ} loosely to denote an open set  in the equivalence class.

Let $\Gamma$ be a finitely generated group, 
$G=\mathrm{SL}_{n+1}(\mathbb R)$ with Lie algebra
$\mathfrak g= \mathfrak{sl}_{n+1}(\mathbb R)$, and 
let $\rho\colon \Gamma\to G$ be a semi-simple representation.
% (eg, with closed orbit under conjugacy). 
Let $G_\rho=\mathrm{Stab}_G(\rho)$ denote the stabilizer
of $\rho$ in $G$ by the action by conjugation.
For an analytic variety  ${\mathcal S}$ where
${G_\rho}$ acts,
we shall use the notation
$$G\times_{G_\rho}{\mathcal S}= (G\times{\mathcal S})/{G_\rho}, $$
where $g\in G_\rho$ maps  $(h,s)\in G\times \mathcal S$ to $(h g^{-1}, g s)$.

\begin{Theorem}
\label{Theorem:Slice}
In the previous setting,
% % assume in addition that (the adjoint of) the stabilizer $G_\rho$ acts semi-simply on $\mathfrak g$.
% % Then
there exists a  $G_\rho$-invariant real analytic
subvariety $\mathcal S\subset \hom (\Gamma, G)$ containing $\rho$
% % .
% % germ at $0$, $\mathcal S\subset H^1(\Gamma,G) $,
such that:
\begin{enumerate}[(i)]
 \item
  $T^{\mathrm{Zar}}_\rho \mathcal S\oplus B^1(\Gamma,
  \mathfrak{g}) =
  T^{\mathrm{Zar}}_\rho \hom(\Gamma, G)$. In particular
 $T^{\mathrm{Zar}}_\rho\mathcal S\cong H^1(\Gamma,\mathfrak{g})$.
 
\item If $\hom(\Gamma, G)$ is smooth at $\rho$, then so is $\mathcal S$.
% % (and therefore
% %  $\mathcal S$ is an open subset in $H^1(\Gamma,\mathfrak{g})$).
\item There is a $G$-equivariant map
 $G\times_{G_\rho}{\mathcal S}\to \hom(\Gamma, G)$ that is a bi-analytic equivalence between
 $G\times_{G_\rho}{\mathcal S}$
 and a $G$-saturated neighborhood of $\rho$.
\end{enumerate}
Furthermore, the germ of $\mathcal S$ at  $\rho$ is  $G_\rho$-equivariantly equivalent to
a germ of a subvariety of  $H^1(\Gamma, G)$ at the origin, that is unique up to
 $G_\rho$-invariant bi-analytic equivalence.
\end{Theorem}

% 
% \begin{Definition}
% We say that $G_\rho$ \emph{acts reductively} on $\mathfrak g$ if every $G_\rho$-invariant subspace $V\subset \mathfrak g$
% has a $G_\rho$-invariant complement $W\subset \mathfrak g$, $V\oplus W = \mathfrak g$.
% \end{Definition}

\begin{proof}
Assume first that $\Gamma=F_k=\langle \gamma_1,\dotsc,\gamma_k\mid\rangle$ is a free group of rank $k$. We consider
the algebraic isomorphism
$$
\begin{array}{rcl}
\phi\colon \hom(F_k,G) &  \to & G^k \\
 \rho' & \mapsto & (\rho'(\gamma_1)\rho(\gamma_1^{-1}),
\dotsc ,\rho'(\gamma_k)\rho(\gamma_k^{-1}))
\end{array}
$$
and the map
$$
\begin{array}{rcl}
f \colon G &  \to &
\mathfrak{g}
=\mathfrak{sl}_{n+1}(\mathbb R) \\
A & \mapsto & A-\frac{\mathrm{trace}(A)}{n+1}\mathrm{Id}
\end{array}
$$
We are interested in the composition
\begin{equation}
 \label{eqn:Phi}
 \Phi= ( f, \dotsc, f)\circ \phi\colon \hom(F_k,G)\to \mathfrak{g}^k.
\end{equation}
By construction $\Phi$ is $G_\rho$-equivariant, $\Phi(\rho)=0$,
and its tangent map at $\rho$
is the isomorphism
$$
\begin{array}{rcl}
\Phi_*\colon Z^1( F_k,\mathfrak{g}) & \to & \mathfrak g^k \\
 d & \mapsto & (d(\gamma_1),\ldots, d(\gamma_k)).
\end{array}
$$
We shall consider an open neighborhood $U\subset  \hom(F_k,G)$
of $\rho$
such that
$\Phi$ defines a bi-analytic map with
$\Phi(U)\subset \mathfrak{g}^k$, a neighborhood of the origin.

The stabilizer $G_\rho$ is easily determined, as $\rho$
is direct sum of simple representations, and one checks that the adjoint representation
of $G_\rho$ in  $\mathrm{End}(\mathfrak{g})$ is also  is semi-simple.
Hence
% Since $G_\rho$ acts reductively on $\mathfrak{g}^k$,
there exists
a $G_\rho$-invariant linear subspace $H\subset \mathfrak{g}^k$ such that
$\Phi_*(B^1 ( F_k,\mathfrak{g})) \oplus H= \mathfrak{g}^k$. In particular
$H\cong H^1 ( F_k,\mathfrak{g})$.

Define the slice as the union of $G_\rho$-translates:
$$
{\mathcal S}_{F_k}= G_\rho\cdot (U\cap \Phi^{-1}(H) )=\bigcup_{g\in G_\rho}
g\cdot  (U\cup \Phi^{-1}(H) )
$$
As $\Phi$ is $G_\rho$-equivariant  and $H$ is
$G_\rho$-invariant, we have the inclusion
$$
{\mathcal S}_{F_k}\cap U= \big( G_\rho \cdot (U\cap \Phi^{-1}(H) )\big)
\cap U
\subset \big( G_\rho \cdot \Phi^{-1}(H) \big)\cap U = \Phi^{-1}(H) \cap U.
$$
Hence
$$
{\mathcal S}_{F_k}\cap U
= \Phi^{-1}(H) \cap U.
$$
In particular  ${\mathcal S}_{F_k}$ is an analytic subvariety that satisfies (i) and (ii) for a free group.

To show that ${\mathcal S}_{F_k}$ satisfies (iii),
notice that the natural map
$
G\times {\mathcal S}_{F_k}\to \hom (F_k, G)
$
is a submersion (by the choice of $H$) and it factors to
$$
\Psi\colon G\times_{G_\rho}{\mathcal S}_{F_k}\to \hom (F_k, G).
$$
The group ${G_\rho}$ acts freely and properly on  $G\times{\mathcal S}_{F_k}$
and $G\times_{G_\rho}{{\mathcal S}_{F_k}}$ is an analytic manifold.
By the submersion theorem,
$\Psi$ defines a bi-analytic map between
$V\subset G\times_{G_\rho}{\mathcal S}_{F_k}$
a neighborhood of the class $G_\rho\times \{0\}$ and $U$,
the neighborhood of  $\rho$.
The image of $\Psi$ is the the union of $G$-iterates
$G \cdot  U=\bigcup_{g\in G} g\cdot  U$,
which is open and saturated, and $\Psi$ is locally bi-analytic by equivariance, in particular it is open.
It remains to check that   $\Psi$ is injective: let
$x, y\in G\times_{G_\rho}{\mathcal S}_{F_k}$ satisfy
$\Psi(x)=\Psi(y)$. By $G$-equivariance we may assume
$\Psi(x)=\Psi(y)\in U$.
From the construction of $\Psi$ as a
a submersion, $U$ is contained in the union of orbits of
$\mathcal{S}_{F_k}\cap U$ and acting by $G$ we may even assume that
$\Psi(x)=\Psi(y)\in {\mathcal S}_{F_k}\cap U$.
As $\Psi(x)\in {\mathcal S}_{F_k}\cap U$ and $\Psi$ is
locally by-analytic,
$x$ is the class of $e\times \Psi(x)\in G \times{\mathcal S}_{F_k}   $
modulo $G_\rho$. Since $\Psi(x)=\Psi(y)$,   $x=y$.
This proves (iii) for a free group.

For a group with a finite presentation
$\Gamma=\langle \gamma_1,\ldots,\gamma_k
\mid (r_j)_{j\in J}\rangle$, let $F_k$ denote the group freely generated by $
\gamma_1,\ldots,\gamma_k$. The projection $F_k\to\Gamma$
induces an inclusion of varieties of representations
 $$
 \hom (\Gamma,G)\hookrightarrow \hom(F_k,G),
 $$
 which in its turn induces an inclusion of Zariski tangent spaces
 $$
Z^1(\Gamma,\mathfrak{g})\subset
Z^1(F_k,\mathfrak{g}).
 $$
 Let $\mathcal S_{F_k}  \subset \hom(F_k,G)$ denote
 the slice for a free group constructed above.
 We consider the intersection
 $$
 \mathcal S_\Gamma= \mathcal S_{F_k}\cap  \hom (\Gamma,G).
 $$
It satisfies  properties (i), (ii) and  (iii)  of a slice,
in particular
$$
T^{\mathrm{Zar}}_{\rho}\mathcal S_\Gamma =
T^{\mathrm{Zar}}_{\rho}\mathcal S_{F_k}\cap
Z^1(\Gamma,\mathfrak{g})\cong H^1(\Gamma,\mathfrak{g}),
$$
because
$B^1(\Gamma,\mathfrak{g})=B^1(F_n,\mathfrak{g})$.

We next construct a $G_\rho$ equivariant bi-analytic equivalence between
$\mathcal S_\Gamma $ and a germ $\mathcal S$ in
$T^{\mathrm{Zar}}_{\rho}\mathcal S_\Gamma$.
At the beginning of the proof for $F_k$, we chose
$H\subset \mathfrak{g}^k$
a $G_\rho$-invariant complement to the space of coboundaries;
again by semi-simplicity of $G_\rho$, we find $G_\rho$-invariant complement $L$:
\begin{equation}
 \label{eqn:complementLL}
H= T^{\mathrm{Zar}}_{\rho}\mathcal S_\Gamma\oplus L
\end{equation}
Then the projection in~\eqref{eqn:complementLL}  of $S_\Gamma  $ to
$T^{\mathrm{Zar}}_{\rho}\mathcal S_\Gamma\cong
H\cap Z^1(\Gamma, \mathfrak{g})$ is the required subset $\mathcal S$.
By construction, this map is $G_\rho$-equivariant, we claim that it
is a bi-analytic equivalence
between $\mathcal  S_\Gamma$ and $\mathcal S$. When $\mathcal
S_\Gamma$ is non-singular the claim is clear from standard arguments,
because it is a projection to the tangent space of a non-singular
subvariety. When $\mathcal  S_\Gamma$ is singular, we embed $\mathcal
S_\Gamma$ in an analytic subvariety $M$ that it is non-singular and
has the same Zariski tangent space at $\rho$. Namely,
we use the argument in \cite[Theorem~V.A.14]{GunningRossi}
to find a real analytic
germ $\mathcal  S_\Gamma\subset M\subset \mathcal S_{F_n}$
with $T_\rho \mathcal S_{\Gamma}=T_\rho M$. We notice that
the projection defines a bi-analytic equivalent between $M$ and
the germ of  $T_\rho\mathcal  S_\Gamma=T_\rho M$ at $\rho$.

Next we prove uniqueness. Two presentations of $\Gamma$ yield two surjections $p_1\colon F_{k_1}\to\Gamma$ and 
$p_2\colon F_{k_2}\to\Gamma$. Let $f\colon  F_{k_1}\twoheadrightarrow F_{k_2}$ and $g\colon  F_{k_2}\twoheadrightarrow F_{k_1}$ 
be lifts of the identity of $\Gamma$:
\begin{center}
\begin{tikzcd}[column sep=small]  F_{k_1}\arrow[two heads]{dr}[left]{p_1}\arrow[rr, "\underset{\phantom{.}}g"]  &&   F_{k_2}\arrow[ll, shift right, 
"\overset{\phantom{.}}{f}"]\arrow[two heads, ld, "p_2"]  & \\
&  \Gamma  &      \end{tikzcd} 
\end{center}
% This induces a commutative diagram of varieies of representations:
% \begin{center}
% \begin{tikzcd}[column sep=small]  \hom(F_{k_1},G)\arrow[two heads]{dr}[left]{p_1}\arrow[rr, "\underset{\phantom{.}}g"]  &&   
% \hom(F_{k_2},G)\arrow[ll, shift right, 
% "\overset{\phantom{.}}{f}"]\arrow[two heads, ld, "p_2"]  & \\
% &  \hom(\Gamma,G)  &      \end{tikzcd} 
% \end{center}
Following the previous construction, this yields two analytic germs of subvarieties 
$$
{\mathcal S}_i\subseteq M_i\subset V_i\subset \mathfrak g^{k_i},\qquad\textrm{ for }i=1,2,
$$
such that $T_{0}^{\mathrm{Zar}}{\mathcal S}_i=T_{0}^{\mathrm{Zar}} M_i\cong H^1(\Gamma,\mathfrak g)$, with $M_i$ non-singular and $V_i$
an open neighborhood of $0$.
% as in \eqref{eqn:V}.
The induced maps $f^* \colon V_2\to V_1$ and $g^* \colon V_1\to V_2$  are  analytic and satisfy
$$ f^*\circ g^*\vert_ {S_1}=\mathrm{Id}_{{\mathcal S}_1}
\qquad  g^*\circ f^*\vert_ {{\mathcal S}_2}=\mathrm{Id}_{S_2}. $$
Therefore the tangent morphisms $d f^*$ and $d g^*$ satisfy 
$$ df^*\circ dg^*\vert_ {T^{\mathrm{Zar}}_0 {\mathcal S}_1}=
\mathrm{Id}_{T^{\mathrm{Zar}}_0 {\mathcal S}_1} \qquad
dg^*\circ df^*\vert_ {T^{\mathrm{Zar}}_0{\mathcal S}_2}=
\mathrm{Id}_{T^{\mathrm{Zar}}_0 {\mathcal S}_2}. $$
In particular $f^* (M_1)$ is a non-singular analytic germ in $\mathfrak g^{k_2}$
and we may chose $M_2=f^* (M_1)$. 
Thus the pairs of germs at the origin
$({\mathcal S}_1,M_1)$ and $({\mathcal S}_2,M_2)$  are equivalent, and uniqueness follows.
\end{proof}

\begin{Corollary}
\label{Corollary:quotient}
Under the hypothesis of the theorem,  a neighborhood of $[\rho]$ in $X(\Gamma, G)$ is homeomorphic
 to $S/\!/{G_\rho}$.
\end{Corollary}

When $G_\rho$ is trivial, eg when $\rho$ is $\mathbb{C}$-irreducible, this provides the natural analytic structure in a neighborhood of the character of $\rho$.
% %  

\subsection{Decomposing the Lie algebra}
\label{Section:Decomposition}

Before proving Theorem~\ref{Theorem:orientable},
we discuss some results
related to the decomposition of the Lie algebra.

Let $\mathcal O^2$ be a compact and orientable 2-orbifold,
with fundamental group
$\Gamma=\pi_1(\mathcal O^2)$. Set
$ G_0=\mathrm{SL}_{n}(\mathbb R)$,
$G=\mathrm{SL}_{n+1}(\mathbb R)$, and view $G_0$ as a subgroup of $G$:
$$
\begin{array}{rcl}
 G_0 & \hookrightarrow & G\\
 A & \mapsto &\begin{pmatrix}
               A & \\ & 1
              \end{pmatrix}
\end{array}
$$
Let $\mathfrak g$ and $\mathfrak g_0$ denote the corresponding Lie algebras. We have a direct sum of
$G_0$-modules
\begin{equation}
\label{eqn:directsum}
\mathfrak g= \mathfrak g_0\oplus \mathbf m_{\mathsf c}\oplus
\mathbf m_{\mathsf r}\oplus \mathbf d
\end{equation}
where  $\mathbf m_{\mathsf c}$ is the subspace of $\mathfrak g$
with vanishing entries away from the last column,
$\mathbf m_{\mathsf r}$ away from the last row, and
 $\mathbf d\cong \mathbb R$ is the subspace of diagonal matrices
 that commute with every element in $\mathfrak g_0$.
As $G_0$-modules,
  $\mathbf m_{\mathsf c}\cong \mathbb R^n_{G_0}$ (eg  $\mathbb R^n$ as  module
by the standard action of $G_ 0$),
$\mathbf m_{\mathsf r}= (\mathbf m_{\mathsf c})^*$  (the  contragredient of $\mathbf m_{\mathsf c}$,
which is also its dual),
and $\mathbf d\cong \mathbb R$ is the trivial module.
% Here $\mathbf d$ stands for diagonal and the subindixes
% $\mathsf c$ and $\mathsf r$ for column and row, as the entries
% of matrices in $\mathbf m_{\mathsf c}$
% vanish away from the last column, and last row for
% $\mathbf m_{\mathsf r}$.

In terms of elementary matrices, if $e_{i,j}$ denotes  the elementary matrix with $1$ in the $i$-th row and $j$-th column, and 0 in the other entries, then as vector spaces we have:
\[
\begin{aligned}
 \mathfrak g_{0}=& \langle e_{i,j}\mid 1\leq i,j\leq n
 \rangle\cap\mathfrak g \\
 \mathbf m_{\mathsf r} = & \langle e_{{n+1}, j}\mid 1\leq j\leq n
  \rangle
  \end{aligned}
 \qquad
\begin{aligned}
 \mathbf m_{\mathsf c} = & \langle e_{i,{n+1}}\mid 1\leq i\leq n
 \rangle
\\
   \mathbf d = & \langle e_{1,1}+\cdots +e_{n,n}-
 n \,e_{n+1,n+1}\rangle
\end{aligned}
\]
The $G_0$-modules $\mathfrak g_0$ and $\mathbf d$ are both self-dual: the Killing form on $\mathfrak g$ restricts
to a non-degenerate form in both $\mathfrak g_0$ and $\mathbf d$  (a multiple of the Killing form of
$\mathfrak g_0$, and a multiple of the isomorphism
$\mathbf d\otimes \mathbf d\cong \mathbb R\otimes \mathbb R\cong \mathbb R$).

The modules $\mathbf m_{\mathsf c}$ and $\mathbf m_{\mathsf r}$ are dual
from each other. The Killing form $B$
on $\mathfrak g$ restricts to a perfect pairing
$B\colon \mathbf m_{\mathsf c}\times \mathbf m_{\mathsf r}\to\mathbb R$ which is $G_0$-invariant. In terms of representations, the adjoint action on
$\mathfrak g$ induces the $\Gamma$-module structures:
$$
\mathbf m_{\mathsf c}\cong \mathbb R^n_\rho \qquad \textrm{ and }
\qquad
\mathbf m_{\mathsf r}\cong \mathbb R^n_{\rho^*}
$$
where $\rho^*$ denotes the contragredient representation,
defined by $\rho^*(\gamma)=(\rho(\gamma)^{-1})^t$,
$\forall\gamma\in\Gamma$. The restriction of the Killing form is also $\Gamma$-invariant perfect pairing
$B\colon \mathbf m_{\mathsf r}\times \mathbf m_{\mathsf c}\to\mathbb R$,
that can be seen as a multiple of  canonical pairing
$\mathbb R^n_\rho \times \mathbb R^n_{\rho^*} \to\mathbb R$,
after the identifications $\mathbf m_{\mathsf c}\cong \mathbb{R}_{\rho}$ and
 $\mathbf m_{\mathsf r}\cong \mathbb{R}_{\rho^*}$.
 It induces  two pairings in cohomology
\begin{align*}
% \label{equation:pairingmpm}
 &H^1(\Gamma, \mathbf m_{\mathsf r})\times
H^1(\Gamma, \mathbf m_{\mathsf c})\to H^2(\Gamma, \mathbb R),
\\
 &H^1(\Gamma, \mathbf m_{\mathsf c})\times
H^1(\Gamma, \mathbf m_{\mathsf r})\to H^2(\Gamma, \mathbb R)
\end{align*}
respectively
defined at the level of 1-cocycles by
\begin{align*}
B(z_{\mathsf r}\cup z_{\mathsf c}) (\gamma_1,\gamma_2)=
z_{\mathsf r}(\gamma_1)^t \rho(\gamma_1)z_{\mathsf c}(\gamma_2),
% \quad \textrm{ and }\quad
\\
B(z_{\mathsf c}\cup z_{\mathsf r}) (\gamma_1,\gamma_2)=
z_{\mathsf c}(\gamma_1)^t \rho^*(\gamma_1)z_{\mathsf r}(\gamma_2),
\end{align*}
following \eqref{eqn:cup}. We remark two further properties.

\begin{Remark}
\label{Remark:PDSkS}
\begin{enumerate}[(a)]
 \item (Non-degeneracy.) When $\Gamma$ is the fundamental group of a closed,
orientable, and aspherical 2-orbifold $\mathcal O^2$,
these pairings in cohomology are non-degenerate by Poincar\'e
duality, stated in Theorem~\ref{TheoremPD}.

 \item (Skew-symmetry.) For $\theta_{\mathsf r}\in H^1(\Gamma, \mathbf m_{\mathsf r})$
and $\theta_{\mathsf c}\in H^1(\Gamma, \mathbf m_{\mathsf c})$,
$
B(\theta_{\mathsf c}\cup \theta_{\mathsf r} )=
-B( \theta_{\mathsf r}\cup \theta_{\mathsf c} )
$, by Lemma~\ref{lemma:symmetry}.
\end{enumerate}
\end{Remark}

% %
% % *************
% %
% % We view a couple of computations in cohomology related to the
% % direct sum in \eqref{eqn:directsum}.
% %
% % \begin{Lemma}
% % \label{Lemma:skewsymmetric}
% % For $\theta_{\mathsf r}\in H^1(\Gamma, \mathbf m_{\mathsf r})$
% % and $\theta_{\mathsf c}\in H^1(\Gamma, \mathbf m_{\mathsf c})$,
% % $
% % B(\theta_{\mathsf c}\cup \theta_{\mathsf r} )=
% % -B( \theta_{\mathsf r}\cup \theta_{\mathsf c} )
% % $.
% % \end{Lemma}
% %
% % \begin{proof} Let $z_{\mathsf r} \in Z^1(\Gamma,\mathbf m_{\mathsf r})$
% % be a cocycle representing $\theta_{\mathsf r}$ and
% % $z_{\mathsf c} \in Z^1(\Gamma,\mathbf m_{\mathsf c})$
% % be a cocycle representing $\theta_{\mathsf c}$.
% %  Consider the 1-cochain $c\colon \Gamma\to\mathbb R$ defined by
% %  $c(\gamma)=z_{\mathsf r}(\gamma)^tz_{\mathsf c}(\gamma)=
% %  B ( z_{\mathsf r}(\gamma), z_{\mathsf c}(\gamma) )$, for $\gamma\in \Gamma$. Then, by an elementary computation,
% %  $
% %  \delta^1 c= B( z_{\mathsf r}\cup z_{\mathsf c})+B(z_{\mathsf c}\cup z_{\mathsf r})
% %  $.
% % \end{proof}

In the formulation of the next lemma we require
\begin{equation}
\label{eqn:proj}
\pi_{\mathsf r}:\mathfrak{g}\to\mathbf m_{\mathsf r}
\quad\textrm{ and } \quad
\pi_{\mathsf c}:\mathfrak{g}\to\mathbf m_{\mathsf c}
\end{equation}
the projections from the direct sum \eqref{eqn:directsum}.

\begin{Lemma}
 \label{Lemma:cupprodcut}
 Let $\mathcal O^2$ be a compact orientable and hyperbolic 2-orbifold, and $\Gamma=\pi_1(\mathcal O^2)$.
 For a $\mathbb C$-irreducible representation
$\rho\colon \Gamma\to G_0$:
\begin{enumerate}[(i)]
 \item   If $\partial  \mathcal O^2\neq \emptyset$, then $H^2(\Gamma, \mathfrak g)=0$.

 \item  If $\partial  \mathcal O^2= \emptyset$, then
 $H^2(\Gamma, \mathfrak g)= H^2(\Gamma, \mathbf d)\cong H^2(\Gamma, \mathbb R)=\mathbb R$.

 In addition, for $\theta\in H^1(\Gamma, \mathfrak g)$,
 $$
 [\theta\cup\theta]=
 c_n \, B( \pi_{\mathsf r}^*(\theta)\cup \pi_{\mathsf c}^* (\theta) )
 = - c_n \, B( \pi_{\mathsf c}^*(\theta)\cup \pi_{\mathsf r}^* (\theta) )
 $$
 for some constant $c_n\neq 0$ depending only on the dimension $n$.
\end{enumerate}
\end{Lemma}

\begin{proof} (i)
When $\partial  \mathcal O^2\neq \emptyset$, then $\mathcal O^2$ has virtually the homotopy type of a graph and therefore
$H^2(\Gamma, \mathfrak g)\cong H^2(\mathcal O^2, \mathfrak g)=0$.

(ii)
When $\partial  \mathcal O^2= \emptyset$,
since $\rho$ is $\mathbb C$-irreducible in $G_0$ the invariant subspaces
$
(\mathfrak g_0)^{\mathrm{Ad}\rho(\Gamma)}
$,
$
(\mathbf m_{\mathsf c})^{\mathrm{Ad}\rho(\Gamma)}
$
and
$
(\mathbf m_{\mathsf v})^{\mathrm{Ad}\rho(\Gamma)}
$
are trivial. In addition, using Poincar\' e duality and the isomorphism between the cohomology of $\mathcal O^2$
and $\Gamma=\pi_1(\mathcal O^2)$, we deduce
$H^2(  \Gamma, \mathfrak g_0 )\cong   H^0( \Gamma, \mathfrak g_0)^*\cong 0$,
$ H^2(\Gamma, \mathbf m_{\mathsf r})\cong
H^0(\Gamma, \mathbf m_{\mathsf r})^*\cong 0$,
and
$ H^2(\Gamma, \mathbf m_{\mathsf c})\cong
H^0(\Gamma, \mathbf m_{\mathsf c})^*\cong 0$.
Then the first assertion in (ii) follows from
the direct sum~\eqref{eqn:directsum}.

From the vanishing of $H^2( \Gamma, \mathfrak g_0 )$,
$ H^2(\Gamma, \mathbf m_{\mathsf r})$
and  $H^2(\Gamma, \mathbf m_{\mathsf c})$, the only relevant term in $ [\theta\cup\theta]$ comes from
the product
$$
[\cdot,\cdot ]\colon (\mathbf m_{\mathsf r}\oplus\mathbf m_{\mathsf c})
\times (\mathbf m_{\mathsf r}\oplus\mathbf m_{\mathsf c})\to \mathbf d.
$$
This motivates the following computation.
For $\gamma_1,\gamma_2\in\Gamma$,
$z_{\mathsf r}\in Z^1(\Gamma, \mathbf m_{\mathsf r})$,
and $z_{\mathsf c}\in Z^1(\Gamma, \mathbf m_{\mathsf c})$,
\begin{multline*}
\left[
\left(
\begin{matrix}
 \text{\Large 0} & z_{\mathsf c}(\gamma_1) \\
 z_{\mathsf r}^t(\gamma_1) & 0
\end{matrix}
\right)
,
\left(
\begin{matrix}
 \text{\Large 0}  & \rho(\gamma_1)z_{\mathsf c}(\gamma_2) \\
 z_{\mathsf r}^t(\gamma_2)\rho(\gamma_1^{-1}) & 0
\end{matrix}
\right)
\right]=
\\
\left(
\begin{matrix}
 \text{\Large *} & 0 \\
 0 & z_{\mathsf r}^t(\gamma_1)\rho(\gamma_1)z_{\mathsf c}(\gamma_2)-
 z_{\mathsf r}^t(\gamma_2)\rho(\gamma_1^{-1})z_{\mathsf c}(\gamma_1)
\end{matrix}
\right)
=
\\
\left(
\begin{matrix}
 \text{\Large *} & 0 \\
 0 & z_{\mathsf r}^t(\gamma_1)\big(\rho(\gamma_1)z_{\mathsf c}(\gamma_2)\big)-
 z_{\mathsf c}^t(\gamma_1)\big(\rho^*(\gamma_1)z_{\mathsf r}(\gamma_2)\big)
\end{matrix}
\right).
\end{multline*} The assertion follows from this formula,
Remark~\ref{Remark:PDSkS}, and Lemma~\ref{lemma:symmetry}.
\end{proof}

\subsection{Prof of Theorem~\ref{Theorem:orientable} }
\label{Section:ProofMT}

The proof is  based on the slice of Section~\ref{Section:slice} and the results on the Lie algebra and cup products in Section~\ref{Section:Decomposition}.

\begin{proof}[Proof of Theorem~\ref{Theorem:orientable}]

We aim to apply Theorem~\ref{Theorem:Slice}, so we  describe
the stabilizer of $\rho$  in $G$:
$$ 
G_{\rho} =\{ \mathrm{diag}(t,\ldots, t, t^{-n})\mid
t\in\mathbb{R}-\{0\}.
   \}
$$
The identity component of  $G_{\rho}$ is
$\exp(\mathbf d)=
\{\mathrm{diag}(e^\lambda,\ldots, e^\lambda, e^{-n\lambda})\mid \lambda\in\mathbb R\}$. 
The group $G_{\rho}$ is generated by its identity component and the matrix
\begin{equation}
	\label{eqn:matrix}
    \begin{cases}
		\mathrm{diag}(-1,\dotsc, -1, -1) & \textrm{if } n+1 \textrm{ is even,} \\
		\\
		\mathrm{diag}(-1,\dotsc,-1,+ 1) & \textrm{if } n+1 \textrm{ is odd.}
	\end{cases}
\end{equation}
%\begin{cases}
%	\mathrm{diag}(-1,\dotsc, -1, -1) & \textrm{if } n+1 \textrm{ is even,} \\
%	\\
%	\mathrm{diag}(-1,\dotsc,-1,+ 1) & \textrm{if } n+1 \textrm{ is odd.}
%\end{cases}
%$$
% %
% % , and check that it acts semi-simply (eg reductively) on $\mathfrak g$:
%$$
%G_{\rho}=
%\begin{cases}
%\{ \pm\mathrm{diag}(e^\lambda,\ldots, e^\lambda, e^{-n\lambda})\mid \lambda\in\mathbb R\} & %\textrm{if } n+1 \textrm{ is even} \\
%\\
%\{ \mathrm{diag}(\epsilon e^\lambda,\ldots,\epsilon  e^\lambda, e^{-n\lambda})\mid \epsilon=\pm %1, \ \lambda\in\mathbb R\} & \textrm{if } n+1 \textrm{ is odd}
%\end{cases}
%$$
%In both cases the identity component of  $G_{\rho}$ is $\exp(\mathbf d)$. The group $G_{\rho}$ is generated by its identity component and the matrix
%$\mathrm{diag}(-1,\dotsc,-1,\pm 1)$ (the sign depends on the parity of
%$n+1$, so that the determinant is $1$).
%
The stabilizer
 $G_\rho$ preserves the decomposition \eqref{eqn:directsum}
and acts trivially on  $\mathfrak g_0$ and $\mathbf d$.
A matrix $
\mathrm{diag}(t,\dotsc, t, t^{-n})
$ acts  by multiplication by
 $t^{n+1}$ on $\mathbf m_{\mathsf c}$, and by  $t^{-(n+1})$
on $\mathbf m_{\mathsf r}$.
In particular we must take into account the parity of $n+1$, as the matrix
in \eqref{eqn:matrix} acts nontrivially  on $\mathbf m_{\mathsf c}$ and 
$\mathbf m_{\mathsf r}$ iff $n+1$ is odd. 
%
%
%
%For $n+1$ even, $\mathrm{diag}(-1,\dotsc,- 1,-1)$ acts trivially on both
%$\mathbf m_{\mathsf c}$ and  $\mathbf m_{\mathsf r}$, but for
%$n+1$ odd, $\mathrm{diag}(-1,\dotsc,-1,1)$ acts by multiplication by $-1$ on
%$\mathbf m_{\mathsf c}$ and  $\mathbf m_{\mathsf r}$.
It follows that  $G_{\rho}$
acts semi-simply on~$\mathfrak{g}$.

The induced action in cohomology of  $G_{\rho}$
preserves the decomposition
$$
H^1(\Gamma,\mathfrak{g}) = H^1(\Gamma,\mathfrak g_0)\oplus H^1(\Gamma,\mathbf m_{\mathsf c})\oplus H^1(\Gamma,
\mathbf m_{\mathsf r})\oplus H^1(\Gamma,\mathbf d)
$$
and the action on each cohomology group is the same as in the corresponding coefficients.
% It is
% trivial on
%  both
% $ H^1(\Gamma,\mathfrak g_{0})$ and $ H^1(\Gamma,\mathbf d)$,
% multiplication
% by $e^{(n+1)\lambda}$ on $H^1(\Gamma,\mathbf m_{\mathsf c})$, and by  $e^{-(n+1)\lambda}$
% on $H^1(\Gamma,\mathbf m_{\mathsf r})$.
%
We set  $p=\dim H^1(\Gamma,\mathfrak g_0)$,
$d=\dim H^1(\Gamma, \mathbb R^n_\rho)
=\dim H^1(\Gamma, \mathbf m_{\mathsf c})$, and
$b=\dim H^1(\Gamma, \mathbf d)=\dim  H^1(\Gamma, \mathbb R)$.

Let $\mathcal S_\Gamma\subset\hom(\Gamma,G)$ be the slice of  Theorem~\ref{Theorem:Slice}, so
% By  Proposition~\ref{Theorem:Slice} a
 a neighborhood of $[\rho]$ is homeomorphic to the germ
$
\mathcal S_\Gamma/\!/G_\rho.
$
%
%
% Next we describe the action  of $G_\rho$ on $H^1(\Gamma,\mathfrak{g})$.
% From the decomposition \eqref{eqn:directsum}, $H^1(\Gamma,\mathfrak{g})$
% is the direct sum:
% $$
% H^1(\Gamma,\mathfrak{g}) = H^1(\Gamma,\mathfrak g_0)\oplus H^1(\Gamma,\mathbf m_{\mathsf c})\oplus H^1(\Gamma,
% \mathbf m_{\mathsf r})\oplus H^1(\Gamma,\mathbf d)
% $$
% Thus an element $\pm
% \mathrm{diag}(e^\lambda,\ldots, e^\lambda, e^{-n\lambda})$ acts trivially on both
% $ H^1(\Gamma,\mathbf m_{\mathsf c})$ and $ H^1(\Gamma,\mathbf d)$, by multiplication
% by $e^{(n+1)\lambda}$ on $H^1(\Gamma,\mathbf m_{\mathsf c})$, and by  $e^{-(n+1)\lambda}$
% on $H^1(\Gamma,\mathbf m_{\mathsf r})$.
% Recall that  $d=\dim H^1(\Gamma, \mathbb R^n_\rho)
% =\dim H^1(\Gamma, \mathbf m_\rho)$.
%
% \bigskip
%
To describe $\mathcal S_\Gamma$ we distinguish two cases.
Assume first that $\partial\mathcal O^2\neq\emptyset$, so
$\hom(\Gamma, G)$ is
 smooth at $\rho$,
because $H^2(\mathcal O^2,\mathfrak g)=0$
(Proposition~\ref{Proposition:Goldman}).
So by    Theorem~\ref{Theorem:Slice}
the germ of $\mathcal S_\Gamma$ at $\rho$ is $G_\rho$-equivariantly
by-analytic to the germ at the origin of $T^{\textrm{Zar}}_0\mathcal S_\Gamma =  H^1(\Gamma,\mathfrak{g})$.
%
%
%
% Since $\mathcal O^2$ is an orbifold, $\Gamma=\pi_1(\mathcal O^2)$ may be non-free, but $\hom(\Gamma, G)$ is
% still smooth at $\rho$,
% because $H^2(\mathcal O^2,\mathfrak g)=0$
% (Proposition~\ref{Proposition:Goldman}).
% So by Proposition~\ref{Theorem:Slice} and Corollary~\ref{Corollary:quotient}, a neigborhood of $[\rho]$ is homeomorphic to
% $$
% H^1(\Gamma,\mathfrak{g})/\!/G_\rho.
% $$
Thus, when $n+1$ is even, a neighborhood of $\rho$ is homeomorphic to $U/\!/\mathbb R$
where
$$
U\subset \mathbb{R}^{p+b}\oplus  (\mathbb R^{d}\oplus \mathbb R^{d})
$$
is a neighborhood of the origin
and  the action of $t\in\mathbb R$ on $(x,y,z)\in U\subset  \mathbb R^{p+b}\oplus \mathbb R^{d}\oplus \mathbb R^{d}$ is given by
$(x,y, z)\mapsto (x, e^ty, e^{-t} z)$.
When $n+1$ is odd, we need to further quotient by the action of  the antipodal on $\mathbb R^{d}\oplus \mathbb R^{d}$.

Define
\begin{equation}
 \label{eqn:M}
\mathcal M= \{(x,y,z)\in\mathbb R^{p+b}\oplus \mathbb R^{d}\oplus \mathbb R^{d}\mid
|y|=|z|\}.
\end{equation}
Then:
$$
U/\!/\mathbb R\cong \mathcal M \cap U
$$
because $\mathcal M$ intersects precisely once each closed orbit (at the point that minimizes the distance to
$\mathbb{R}^{p+b}\times\{0,0\}$, eg an elementary example of Kempf-Ness). Therefore a neighborhood of the character is homeomorphic to
$$
\begin{cases}
 \mathcal M \cap U & \textrm{when } n+1 \textrm{ is even} \\
 \mathcal M \cap U/\!\sim & \textrm{when } n+1 \textrm{ is odd, } \\
\end{cases}
$$
where $\sim$ denotes the antipodal relation on the last $2 d$ coordinates.

\medskip

Next we deal with the case  $\partial\mathcal O^2=\emptyset$.
Here we use Theorem~\ref{Theorem:Quadratic},
the result of \cite{GoldmanMaryland,GoldmanMillson,Simpson} on quadratic singularities.
We describe a neighborhood using a power expansion.
Fix a norm in $\mathfrak g$, say
$|a|=\mathrm{trace}(a^t a)$, hence a norm in $\mathfrak g^k$ and in
$Z^1(\Gamma, \mathfrak g)$.
By Theorem~\ref{Theorem:Quadratic}, a neighborhood   $V\subset \hom(\Gamma, G)$ of $\rho$  can be described as
\begin{equation}
	\label{eqn:V}
	V\cong\{ \varepsilon z+ O(\varepsilon^2)\mid
	\varepsilon\in [0,\varepsilon_0), \
	z\in Z^1(\Gamma, \mathfrak g)\mid
	[z\cup z]\sim 0, |z| = 1\}.
\end{equation}
Now we follow the construction of the slice $\mathcal S$ in the
proof of Theorem~\ref{Theorem:Slice}, in particular we have a surjection of the
free group of rank $k$, $F_k\twoheadrightarrow \Gamma$, an
inclusion of varieties of representations
$\hom(\Gamma, G)\subset \hom(F_k, G)$,
and
a map $\Phi\colon \hom(F_k, G)\to \mathfrak{g}^k\cong
Z^1(F_k,\mathfrak{g})\cong T^{\mathrm{Zar}}_\rho\hom(F_k, G)$ in
\eqref{eqn:Phi} by-analytic in a neighborhood of $\rho$.
The proof of Theorem~\ref{Theorem:Slice}
considers $H\subset
Z^1(F_k,\mathfrak g)$ a $G_\rho$-invariant complement to
$B^1(F_k,\mathfrak g)$ and define the slice
in $\hom(F_k, G)$ as a neighborhood in $\Phi^{-1}(H)$.
The slice in
$\hom(\Gamma, G)$ is its trace in $\hom(\Gamma, G)$:
$S_\Gamma=\Phi^{-1}(H)\cap V$, where $V$ is as in~\eqref{eqn:V}.
% %
% % $V=\Phi^{-1}(H)\cap U$
% %
% % we view $V$  as a germ at the origin of a subvariety
% % $V\subset Z^1(F_k,\mathfrak g)$ satisfying
% % $T_0^{\mathrm{Zar}} V= Z^1(\Gamma,\mathfrak{g})$.
% % The proof of
% % Theorem~\ref{Theorem:Slice} furthermore considers $H\subset
% % Z^1(F_k,\mathfrak g)$ a $G_\rho$-invariant complement to
% % $B^1(F_k,\mathfrak g)$ and defines $\mathcal S_\Gamma=H\cap V$
% %  so that $T_0^{\mathrm{Zar}}
% % \mathcal S_\Gamma
% % \cong H^1(\Gamma, \mathfrak{g})$.
Thus
$$
 \mathcal S_\Gamma \cong\{ \varepsilon z+ O(\varepsilon^2)\mid
\varepsilon\in [0,\varepsilon_0), \
z \in H\cap Z^1(\Gamma,\mathfrak g)
\mid
[z\cup z]\sim 0, |z | = 1\}.
$$
Since $H\cap Z^1(\Gamma,\mathfrak g)\cong H^1(\Gamma,\mathfrak g)$,
and $\mathcal S$ is a projection of $\mathcal S_\Gamma$ to
$H\cap Z^1(\Gamma,\mathfrak g)$:
$$
\mathcal S \cong \{ \varepsilon \theta+ O(\varepsilon^2)\mid
\varepsilon\in [0,\varepsilon_0), \
\theta \in H^1(\Gamma,\mathfrak g)
\mid
[\theta\cup \theta]= 0, |\theta | = 1\}.
$$
Therefore if $\mathcal M\subset H^1(\Gamma, \mathfrak{g})$
is a minimal set as in \eqref{eqn:M} and $(G_\rho)_0\cong\mathbb R^*$
denotes the identity component of $G_\rho$:
$$\mathcal S/\!/(G_\rho)_0 \cong
\mathcal S\cap\mathcal M=\{ \varepsilon \theta+ O(\varepsilon^2)\mid
\varepsilon\in [0,\varepsilon_0), \
\theta\in \mathcal M \mid
[\theta\cup \theta]\sim 0, | \theta | = 1\}.
$$
So $\mathcal S\cap\mathcal M$ is homeomorphic to the cone on
the subset of the unit sphere in $ H^1(\Gamma, \mathfrak g)$ defined by $\mathcal M$ and
the vanishing of the cup product.  
Notice that both the equation defining $\mathcal M$ 
and the cup product are trivial on the spaces
$H^1(\Gamma, \mathfrak{g}_0)$
and $H^1(\Gamma, \mathbf{d})$.
Hence, using the definition of $\mathcal M$ and Lemma~\ref{Lemma:cupprodcut}, the neighborhood in $X(\Gamma, G)$ is homeomorphic
to a neighborhood of the origin in
$$
H^1(\Gamma, \mathfrak{g}_0)\oplus H^1(\Gamma, \mathbf{d})\oplus
\begin{cases}
{\mathcal C} & \textrm{ when } n+1 \textrm{ is even} \\
{\mathcal C}/\!\sim\! & \textrm{ when } n+1 \textrm{ is odd},
\end{cases}
$$
where ${\mathcal C}$ is the cone
$$
{\mathcal C}=
 \{(\theta_{\mathsf c},\theta_{\mathsf r})\in
 H^1(\Gamma, \mathbf m_{\mathsf c})\oplus H^1(\Gamma, \mathbf m_{\mathsf r}) \mid |\theta_{\mathsf c}|^2= |\theta_{\mathsf r}|^2\textrm{ and }
 \theta_{\mathsf c}\cdot \theta_{\mathsf r}=0\}.
$$
% Here the norm uses the identification $\Sigma\cap Z^1(\Gamma,\mathfrak g)\cong H^1(\Gamma, \mathfrak g)$.
The cup product
$$
H^1(\Gamma, \mathbf m_{\mathsf c})\times H^1(\Gamma, \mathbf m_{\mathsf r})\to
H^2(\Gamma, \mathbf d)\cong\mathbb R
$$
is a perfect pairing, by Theorem~\ref{TheoremPD}. Then using elementary linear algebra to combine
the norms with the scalar product, we have that
$$
{\mathcal C}\cong\{(x,y)\in\mathbb R^d\times \mathbb R^d\mid |x|^2=|y|^2,\, x\cdot y=0\}\cong \mathrm{Cone}(UT(S^{d-1})),
$$
and the theorem follows.
\end{proof}

\subsection{Non-orientable orbifolds}
 \label{subsection:nonorientable}
Next we discuss \emph{non orientable} 2-orbifolds.
Instead of working with $\mathrm{SL}_n(\mathbb R)$,
for non-orientable orbifolds it makes sense to work with representations of its  fundamental group in
$$
 \mathrm{SL}^{\pm }_n (\mathbb R)=\{A\in\textrm{GL}_n(\mathbb R)\mid \det (A)=\pm 1\}.
$$

\begin{Definition}
\label{Definitiontp}
 A representation $\rho\colon\pi_1(\mathcal O^2)\to \mathrm{SL}^{\pm }_n (\mathbb R)$ is called
 \emph{type preserving} if it maps  orientation preserving
 elements in $\pi_1(\mathcal O^2)$
 to matrices of determinant $1$ and orientation reversing
 elements to matrices of determinant~$-1$.
\end{Definition}

We consider two natural embeddings of $\mathrm{SL}^{\pm }_n (\mathbb R)$:
\begin{itemize}
 \item \emph{Orientable embedding.} We view every projective transformation of $\mathbb P^{n-1}$ as an orientation preserving transformation of $\mathbb P ^n$:
 $$
 \begin{array}{rcl}
  \mathrm{SL}^{\pm }_n (\mathbb R)& \hookrightarrow &
  \mathrm{SL}_{n+1}(\mathbb R) \\
  A & \mapsto & \begin{pmatrix}
                A & \\ & \det(A)
                \end{pmatrix}
 \end{array}
 $$
 \item \emph{Type preserving.}
 An element preserves the orientation on   $\mathbb P^{n-1}$ if, and only if, it preserves
 the orientation  on  $\mathbb P^{n}$:
 $$
 \begin{array}{rcl}
  \mathrm{SL}^{\pm }_n (\mathbb R)& \hookrightarrow &
  \mathrm{SL}^{\pm}_{n+1} (\mathbb R) \\
  A & \mapsto & \begin{pmatrix}
                 A & \\ & 1
                \end{pmatrix}
 \end{array}
 $$
\end{itemize}
To state the theorem, we introduce the following notation:
\begin{itemize}
 \item $d_{\mathrm{tp}}=\dim H^1(\mathcal O^2,\mathbb R^n_\rho)$, and
 \item   $d_{\mathrm{oe}}=\dim H^1(\mathcal O^2,\mathbb R^n_{\rho\otimes \alpha})$,
where $\alpha(\gamma)=\det\rho(\gamma)\in\{\pm 1\}$ for $\alpha\in\Gamma$.
\end{itemize}
In the type preserving case, we work with $d_{\mathrm{tp}}$
and in the orientable embedding, with  $d_{\mathrm{oe}}$.

\begin{Theorem}
 \label{Theorem:nonorientable}
Let $\mathcal O^2$ be a compact, non-orientable 2-orbifold with $\chi(\mathcal O^2)<0$, let
$\rho\colon\pi_1(\mathcal O^2)\to \mathrm{SL}^\pm_n(\mathbb R)$ be a type preserving  representation, and let
$d_{\mathrm{oe}}\geq 0 $ and $d_{\mathrm{tp}}\geq 0$  be as above.
Assume that $\rho$ restricted to the orientation covering of $\mathcal O^2$
is $\mathbb C$-irreducible.
\begin{enumerate}[(i)]
 \item For the orientable embedding, a neighborhood of the character of $\rho$ in
 $ X(\mathcal O^2, \mathrm{SL}_{n+1}(\mathbb R))$ is
 homeomorphic to
 $$
 \mathbb R^p\times \mathbb R^b\times
 \begin{cases}
 \mathrm{Cone}( S^{d_{\mathrm{oe}}-1}\times S^{d_{\mathrm{oe}}-1})
 & \textrm{if } n+1\textrm{ is even}
 \\
 \mathrm{Cone}( S^{d_{\mathrm{oe}}-1}\times S^{d_{\mathrm{oe}}-1})/\!\sim
 & \textrm{if } n+1\textrm{ is odd}
 \\
 \end{cases}
 $$
 \item For the type preserving embedding, a neighborhood of the character of $\rho$ in
 $ X(\mathcal O^2, \mathrm{SL}^{\pm }_{n+1}(\mathbb R))$ is
 homeomorphic to
 $$
 \mathbb R^p\times \mathbb R^b\times  \mathrm{Cone}( S^{d_{\mathrm{tp}}-1}\times S^{d_{\mathrm{tp}}-1})/\sim
 $$\end{enumerate}
 In both cases
 $X(\mathcal O^2, \mathrm{SL}^\pm_n(\mathbb R) ) $
 corresponds to $\mathbb R^p\times\{0\}\times\{0\}$, and
 $\sim$ denotes the action of the antipodal map in
 $S^{d-1}\times S^{d-1}\subset
 S^{2 d-1}\subset \mathbb R^{2 d}$.
\end{Theorem}

The theorem has the same proof as
Theorem~\ref{Theorem:orientable} in the orientable case with some minor changes.
Notice that by Lemma~\ref{LemmaH20} below,
$H^2(\mathcal O^2,\mathbb R)=0$, so there is no obstruction to integrability and we are
in the same situation as when the orientable orbifold has boundary.
The minor changes in the proof depend on the kind of embedding:
\begin{itemize}
 \item  For the orientable embedding, in
the $\pi_1(\mathcal O^2)$-module $\mathbb R^n$, instead of the action of
$\rho$, we consider the action of $\rho\otimes\alpha$.
% where $\alpha(\gamma)=\det \rho(\gamma)\in\{\pm 1\}$.
\item
 In the orientable embedding we have to care about the parity of $n+1$, but in the type preserving embedding not, because the group
 $\mathrm{SL}^{\pm }_{n+1} (\mathbb R)$ always contains the matrix
 $\mathrm{diag}(-1,\dotsc,-1,1)$.
\end{itemize}

%
%
% . The only difference is
% in
% the $\pi_1(\mathcal O^2)$-module $\mathbb R^n$, instead of the action of
% $\rho$, we consider the action of $\rho\otimes\alpha$,
% where $\alpha(\gamma)=\det \rho(\gamma)\in\{\pm 1\}$.
%  In the type preserving embedding, the stabilizer of $\rho$ in
%  $\mathrm{SL}^{\pm }_{n+1} (\mathbb R)$ is larger than its stabilizer in
%  in  $\mathrm{SL}_{n+1}(\mathbb R)$, it is a degree 2 extension,
%  so we need to add the antipodal map in the quotient.

\begin{Lemma}
\label{LemmaH20}
 Let $\mathcal O^2$ and $\rho$ be as in the theorem.
 Then $H^2(\mathcal O^2,\mathfrak g)=0$.
\end{Lemma}

\begin{proof}[Proof of the lemma]
	Let 
	 $\widetilde{\mathcal O^2}$ denote the orientation covering of $\mathcal O^2$, the cohomology of $\mathcal O^2$ is isomorphic to the invariant subspace of the cohomology of  $\widetilde{\mathcal O^2}$ invariant by the action of 
	 $\mathbb{Z}/2\mathbb{Z}$, the group of deck transformations:
	 $H^2(\mathcal O^2,\mathfrak g)\cong H^2( \widetilde{\mathcal O^2}, 
	 \mathfrak g   )^{\mathbb{Z}/2\mathbb{Z}  } 
	 $. By Lemma~\ref{Lemma:cupprodcut},
	 $H^2( \widetilde{\mathcal O^2},\mathfrak g )=H^2( \widetilde{\mathcal O^2},\mathbf d )\cong 
	 H^2( \widetilde{\mathcal O^2},\mathbb R )\cong \mathbb R  $, and 
	 $\mathbb{Z}/2\mathbb{Z}$ acts by change of sign on $\mathbb{R}$, so the invariant
	 	subspace is trivial (by the same reason why
	 	$H^2( {\mathcal O^2},\mathbb R )=0$). 
%	
% By \cite[Lemma~3.6]{PortiDim}, if $\widetilde{\mathcal O^2}$ denotes the
% orientation covering of $\mathcal O^2$, then
% $$
% \dim H^0(\mathcal O^2,\mathfrak g)+
% \dim H^2(\mathcal O^2,\mathfrak g) =
% \dim H^2(\widetilde{\mathcal O^2},\mathfrak g).
% $$
% In addition, by Poincaré duality:
% $$
% \dim H^0(\widetilde{\mathcal O^2},\mathfrak g)
%=
% \dim H^2(\widetilde{\mathcal O^2},\mathfrak g).
% $$
% Finally,
% as $\rho$ restricted to $\pi_1( \widetilde{\mathcal O^2})$
% is $\mathbb C$-irreducible:
% $$
%  \dim H^0(\mathcal O^2,\mathfrak g)=
% \dim H^0(\widetilde{\mathcal O^2},\mathfrak g).
% $$
%The lemma follows form these formulas.
\end{proof}

\section{Convex projective 2-orbifolds}
\label{Section:proj2}

Let $\mathcal O^2$ be a compact 2-orbifold with
negative Euler characteristic, hence hyperbolic.
The holonomy representation of its
convex projective structure lies in
$\mathrm{PGL}_3(\mathbb R)$, and by Choi-Goldman
\cite{ChoiGoldman} in lies in the same component
in the variety of representations as the holonomy of
its hyperbolic structure.
The holonomy of a hyperbolic structure in $\mathrm{PO}(2,1)$ lifts to
$\mathrm{O}(2,1)$, thus:

\begin{Remark}
\label{Remark:lift}
The projective holonomy of a convex projective structure on a 2-orbifold lifts to
 $\mathrm{SL}^{\pm}_3(\mathbb R)$, or
$\mathrm{SL}_3(\mathbb R)$ when it is orientable.
\end{Remark}

Irreducibility over $\mathbb R$ of the holonomy is well known. Since $3$ is odd, we also have:

\begin{Remark}
The holonomy of a convex projective structure of
$\mathcal O^3$ in  $\mathrm{SL}^{\pm}_3(\mathbb R)$
is $\mathbb C$-irreducible.
\end{Remark}

\subsection{Orientable projective 2-orbifolds}

For $\mathcal O^2$ \emph{orientable} and compact, we consider the following spaces:
\begin{itemize}
  \item The Teichm\"uller space $\mathrm{Teich}(\mathcal O^2)\cong \mathbb R ^t$, which is an open set in
  a component of the variety of characters,
  $X_0(\mathcal O^2, \textrm{SO}(2,1))$
  (the whole component when $\mathcal O^2$ is closed).
   \item The space of  convex projective structures  $\mathrm{Proj}_{\mathrm{cvx}}(\mathcal O^2)$, which is homeomorphic to a cell
   $\mathbb R ^p$, as proved by  Choi-Goldman
   \cite{ChoiGoldman}. The
   projective structures are required to have  principal geodesic boundary components, so $\mathrm{Proj}_{\mathrm{cvx}}(\mathcal O^2)$ is not the whole component  of characters
   $X_0(\mathcal O^2, \mathrm{SL}_3(\mathbb R))$,
   but an open subset (see Remark~\ref{Remark:lift} regarding the lift from $\mathrm{PSL}$ to $\mathrm{SL}$).
   %
%    Notice that as $\mathcal O^2$ is orientable, we may work with $\mathrm{SL}_3(\mathbb R)$ instead of
%    $\mathrm{PGL}(3,\mathbb R)$, because this
%    space of representations contains the hyperbolic holonomies in $\textrm{SO}(2,1)$.
\end{itemize}

For $\mathcal O^2$ compact and orientable, since the standard representation of $\mathrm{SO}(2,1)$
on $\mathbb R^3$
is equivalent to the adjoint
representation in $\mathfrak{so}(2,1)$:
%\begin{multline}
\begin{equation}
	\label{eqn:DimTeich}
t=\dim( \mathrm{Teich}(\mathcal O^2))=
\dim (H^1(\mathcal O^2,\mathbb R^3_\rho))
%=\dim (X_0(\mathcal O^2, \textrm{SO}(2,1)) ) \\ 
= -3\chi(\vert \mathcal O^2\vert ) + 2 c(\mathcal O^2),
\end{equation}
%\end{multline}
%
% we have the following formulas for the dimensions o fthe Teichm\"uller space:
% \[
%  t =  \dim  (X_0(\mathcal O^2, \textrm{SO}(2,1)) ) = -3\chi(\vert \mathcal O^2\vert ) + 2 c(\mathcal O^2)
% \]
% \begin{align*}
%  t = & \dim  (X_0(\mathcal O^2, \textrm{SO}(2,1)) ) = -3\chi(\vert \mathcal O^2\vert ) + 2 c(\mathcal O^2) \\
%  p = & \dim  ( X_0(\mathcal O^2, \textrm{SL}(3,\mathbb R)) ) = -8\chi(\vert \mathcal O^2\vert ) + 4 c_2(\mathcal O^2)+ 2 c(\mathcal O^2)\\
%  b= & \dim (H_1(|\mathcal O^2|,\mathbb R))
% \end{align*}
where $\vert \mathcal O^2\vert $ denotes the underlying surface of the orbifold and $c(\mathcal O^2)$ is the number of cone points of
% the number of cone points and
% $c_2(\mathcal O^2)$ the number of cone points of order 2
 $\mathcal O^2$. See  \cite{Weil} or  \cite{PortiDim}.
Theorem~\ref{Theorem:orientable} immediately yields:

\begin{Corollary}
\label{Corollary:projective}
 Let $\mathcal O^2$ be a compact orientable 2-orbifold,
 with negative Euler characteristic and a convex projective structure. Then
 a neighborhood of the character of its projective holonomy  in
 $ X(\mathcal O^2, \mathrm{SL}_4(\mathbb R))$ is
 homeomorphic to
%  $ (\mathbb R^{p}\times \mathbb R^b\times X, \mathbb R^{p}\times\{0\}\times\{0\}) $, where
 $$
%  \begin{cases}
%         \mathrm{Proj}_{\mathrm{cvx}}(\mathcal O^2)\times \mathbb R^b\times  \mathrm{Cone}( \mathrm{UT}( S^{t-1})) & \textrm{if }\mathcal O^2 \textrm{ is closed}\\
%          \mathrm{Proj}_{\mathrm{cvx}}(\mathcal O^2)\times \mathbb R^b\times
%          \mathrm{Cone}( S^{t-1}\times S^{t-1}) & \textrm{if }\mathcal O^2 \textrm{ has boundary}
%         \end{cases}
 \begin{cases}
        \mathbb R^{p+b}\times  \mathrm{Cone}( \mathrm{UT}( S^{t-1})) & \textrm{if }\mathcal O^2 \textrm{ is closed}\\
         \mathbb R^{p+b}\times
         \mathrm{Cone}( S^{t-1}\times S^{t-1}) & \textrm{if }\mathcal O^2 \textrm{ has boundary}
        \end{cases}
 $$ where
 $t=\dim ( \mathrm{Teich}(\mathcal O^2) )$, $p=\dim ( \mathrm{Proj}_{\mathrm{cvx}}(\mathcal O^2))$
 and $t=\dim ( H^1(\mathcal{O}^2,\mathbb R))$.
%  , $U\cong \mathbb R^p$,
% %  , $t=\dim(H^1(\mathcal O^2,\mathbb R^n_\rho))$,
%  and
%  such that $X(\mathcal O^2, \mathrm{SL}_{n}(\mathbb R) ) $
%  corresponds to $U\times\{0\}\times\{0\}$.
\end{Corollary}

Notice that for  $\mathcal O^2$ compact, orientable and hyperbolic,
by \eqref{eqn:DimTeich}
$t=\dim  (\mathrm{Teich}(\mathcal O^2))=0$ if and only if
$\mathcal O^2$ is a turnover (a 2-sphere with 3 cone points).
In addition, for a turnover $H^1(\mathcal O^2,\mathbb R)=0$.
So in this case we have:

\begin{Corollary}
\label{Corollary:turnover}
  Let  $\mathcal O^2$ be a turnover, every deformation in
  $\mathrm{SL}_4(\mathbb R)$ of its projective holonomy is
  conjugate to a deformation in  $\mathrm{SL}_3(\mathbb R)$.
\end{Corollary}

By \cite{ChoiGoldman}, the space  $\mathrm{Proj}_{\mathrm{cvx}}(\mathcal O^2)$ may be
non-trivial for a turnover. In fact, if all cone points of $\mathcal O^2$ have order at least three, then
$\dim (\mathrm{Proj}_{\mathrm{cvx}}(\mathcal O^2))=2$.
If one of the cone points has order two, then $ \mathrm{Proj}_{\mathrm{cvx}}(\mathcal O^2)$
is a point. By hyperbolicity, at most one of the cone points of the turnover
has order 2.

\subsection{Non-orientable projective 2-orbifolds}

We follow the notation of Subsection~\ref{subsection:nonorientable}. In particular
 let $\alpha\colon\pi_1(\mathcal O^2)\to \{\pm 1\}$
 be the orientation representation: $\alpha(\gamma)=1$
 if $\gamma$ preserves the orientation and $\alpha(\gamma)=-1$ if $\gamma$ reverses it.

 The group of isometries of hyperbolic plane is identified to $\mathrm{O}_0(2,1)$, the components of
 the group of Lorentz transformations of $\mathbb R^2_1$
 that preserve  each leaf of the hyperboloid.

 \begin{Lemma}
 \label{lemma:equiv}
  Let $\rho\colon \pi_1(\mathcal O^2)\to
  \mathrm{O}_0(2,1)$ be the holonomy of a  hyperbolic
  structure.
  Then the representation
 $\rho\otimes \alpha$ is equivalent to the adjoint representation, acting on $\mathfrak so(2,1)$.
 \end{Lemma}

 \begin{proof}
 We have isomorphisms of Lie groups
 $$
 \mathrm{O}_0(2,1)  \cong \mathrm{PO}_0(2,1)\cong \mathrm{PO}(2,1)\cong \mathrm{PSO}(2,1)\cong \mathrm{SO}(2,1).
 $$
 Then the lemma follows from the following two assertions:
 \begin{enumerate}[a)]
  \item The isomorphism $\mathrm{O}_0(2,1) \cong  \mathrm{SO}(2,1)$ consists in multiplying each matrix  by its determinant. Notice also that
  $\alpha(\gamma)=\mathrm{det}(\rho(\alpha))$.
  \item The tautological representation of $\mathrm{SO}(2,1)$
is equivalent to its adjoint action on $\mathfrak{so}(2,1)$.
 \end{enumerate}
Assertion a) is easily checked by looking at the action on Lorentz space. Assertion b) is ``well known'', and follows for instance
from the classification of representations of
$\mathrm{PSL}(2,\mathbb{C})\cong \mathrm{SO}(3,\mathbb{C})$.
%
%
% it is an easy computation to check it on
% one-parameters groups that span the identity component of $\mathrm{SO}(2,1)$, and one element that reverses the orientation of hyperbolic plane.
 \end{proof}

 We view the holonomy of projective structures in
 $\mathrm{SL}^{\pm}_3(\mathbb R)$, as deformations
 of the hyperbolic holonomy in $\mathrm{O}_0(2,1)$,
 in particular type preserving
 (Definition~\ref{Definitiontp}). Recall from Subsection~\ref{subsection:nonorientable}
 that
 $$d_{\mathrm{tp}}=\dim H^1(\mathcal O ^2,\mathbb R^n_\rho)\qquad\textrm{ and }\qquad
 d_{\mathrm{oe}}=\dim H^1(\mathcal O ^2,\mathbb R^n_{\rho\otimes \alpha}).
 $$

 \begin{Lemma}
 \label{lemma:dims}
  Let $\mathcal O^2$ be a compact non-orientable hyperbolic 2-orbifold. Let $\rho\colon \pi_1(\mathcal O^2)\to \mathrm{SL}^{\pm}_3(\mathbb R)$
  be the holonomy of a convex projective structure.
 Then:
 \begin{enumerate}[(a)]
 \item $ H^i(\mathcal O ^2,\mathbb R^n_\rho)
 = H^i(\mathcal O ^2,\mathbb R^n_{\rho\otimes \alpha})=0$, for $i=0,2$.
  \item $d_{\mathrm{oe}}=\dim( \mathrm{Teich}(\mathcal O^2)) $.
  \item $d_{\mathrm{tp}}= \dim (\mathrm{Teich}(\mathcal O^2))-\mathsf{f}(\partial \mathcal O^2)  $.
 \end{enumerate}
where $\mathsf{f}(\partial \mathcal O^2)  $ denotes the number of
components of the boundary that are full 1-orbifolds.
\end{Lemma}

\begin{Definition}
 A \emph{full} 1-orbifold is  the quotient of the real line by the group generated by two different reflections.

\end{Definition}

A full 1-orbifold is closed and non-orientable. Its underlying space is an interval and the isotropy group
of the endpoints of the interval is the cyclic group with two elements (generated by a reflection on the line).
It is the quotient of the circle by a non-orientable involution.

\begin{Corollary}
\label{Corollary:projectiveNO}
Let $\mathcal O^2$ be a compact non-orientable 2-orbifold,
with negative Euler characteristic and a convex projective structure.
\begin{enumerate}[(i)]
	\item For the orientable embedding, a neighborhood of the character of its projective holonomy  in
	$ X(\mathcal O^2, \mathrm{SL}_4(\mathbb R))$ is
	homeomorphic to
	$$
	\mathbb R^{p+b}\times\mathrm{Cone}( S^{t-1}\times S^{t-1}).
	$$
	\item For the type preserving embedding,  a neighborhood of the character of its projective holonomy  in
	$ X(\mathcal O^2, \mathrm{SL}_4^{\pm}(\mathbb R))$ is
	homeomorphic to
	$$
	\mathbb R^{p+b}\times\mathrm{Cone}( S^{t-\mathsf{f}-1}\times S^{t-\mathsf{f}-1})/\!\sim.
	$$
\end{enumerate}	
Here
$t=\dim ( \mathrm{Teich}(\mathcal O^2) )$, $p=\dim ( \mathrm{Proj}_{\mathrm{cvx}}(\mathcal O^2))$, 
 $b=\dim ( H^1(\mathcal{O}^2,\mathbb R))$, and 
 $\mathsf{f}=\mathsf{f}(\partial \mathcal O^2)  $.
\end{Corollary}

\begin{proof}[Proof of Lemma~\ref{lemma:dims}]

Assertion (a) uses the fact that the cohomology of
$\mathcal O^2$ is
the invariant cohomology of its orientation
covering, which is also a projective orbifold.
Both $\rho$ and $\rho\otimes\alpha$ restrict to
the projective holonomy of the orientation
covering, for which the 0th and 2nd cohomology
group with coefficients in $\mathbb R^3_\rho$
vanish, as discussed in
Section~\ref{Section:Decomposition}.

To prove the other statements,
we use Proposition~3.2 of \cite{PortiDim}. Namely, if $K$ is a finite cell decomposition of $\mathcal O^2$,
so that the isotropy group $\mathrm{Stab}(\tilde e)$ is constant on each cell of $K$
(here $\tilde e$ is a lift of $e$ to the universal covering),
then
\begin{equation}
 \label{eqn:twistedEuler}
\sum_i (-1)^i\dim H^i(\mathcal O^2, V)= \sum_{e\textrm{ cell of }K} (-1)^{\dim e} \dim V^{\mathrm{Stab}(\tilde e) }.
 \end{equation}

Assertion (b) holds true for
$\rho\colon\pi_1(\mathcal{O}^2)\to \mathrm{O}_0(2,1)$
the holonomy   of a hyperbolic structure,
because by Lemma~\ref{lemma:equiv}
 $\mathrm{Ad}\rho$
is equivalent to $\rho\otimes \alpha$.  Then the equality of dimensions  holds true in the whole connected component of
the variety of representations by (a) and
\eqref{eqn:twistedEuler}, because the conjugacy
class of the image of a
finite order element does not change along
a connected component.

To prove (c),  we shall compute
the difference  $d_{\mathrm{oe}}-d_{\mathrm{tp}}$.
We aim to apply Assertion~(a) and
Equality~\eqref{eqn:twistedEuler}, so we
compute
 difference
\begin{equation*}
% \label{eqn:diffstab}
 \dim \big( ( \mathbb R^3_\rho   )^{\mathrm{Stab}(\tilde e)} \big)- \dim  \big( ( \mathbb R^3_{\rho\otimes\alpha } )  ^{\mathrm{Stab}(\tilde e)} \big).
\end{equation*}
for each cell $e$ of $K$.
 When the isotropy group of $e$ preserves
 the orientation, then $\alpha$ is trivial on $\mathrm{Stab}(\tilde e)$
 and this difference  vanishes.
 %  In our situation, if the isotropy group of a cell $e$ is trivial or orientation preserving, then
% \begin{equation}
%  \label{eqn:dimsor}
%  \dim  ( \mathbb R^3_\rho   )^{\mathrm{Stab}(\tilde e)}= \dim  ( \mathbb R^3_{\rho\otimes\alpha } )  ^{\mathrm{Stab}(\tilde e)}
% \end{equation}
% because $\alpha$ is trivial on $\mathrm{Stab}(\tilde e)$ under this hypothesis.
When the stabilizer of a cell is non-orientable, it is a group generated by reflections, either  $\mathbb{Z}/2\mathbb{Z}$ or a dihedral
group $D_{2n}$.
For $\mathbb{Z}/2\mathbb{Z}$ the images are generated by:
$$
\rho( \mathbb{Z}/2\mathbb{Z} )= \langle  \left(\begin{smallmatrix} -1 & & \\ & 1 & \\ & & 1 \end{smallmatrix}\right)   \rangle
\quad \textrm{ and }\quad
(\rho\otimes\alpha) ( \mathbb{Z}/2\mathbb{Z} )= \langle    \left(\begin{smallmatrix} 1 & & \\ & -1 & \\ & & -1 \end{smallmatrix}\right)
\rangle.
$$
Hence the dimension of the invariant spaces are
$$
 \dim(\mathbb R^3_{\rho})^{\mathbb{Z}/2\mathbb{Z}}= 2
\quad \textrm{ and }\quad
 \dim(\mathbb R^3_{\rho\otimes\alpha})^{\mathbb{Z}/2\mathbb{Z}}= 1.
$$
For the dihedral group $D_{2n}$, the generators of the image are
\begin{align*}
 \rho( D_{2n} )&= \langle  \left(\begin{smallmatrix} -1 & & \\ & 1 & \\ & & 1 \end{smallmatrix}\right) ,
 \left(\begin{smallmatrix} \cos \alpha &  -\sin\alpha & \\ \sin\alpha & \cos \alpha & \\ & & 1 \end{smallmatrix}\right)
\rangle
\\
(\rho\otimes\alpha) (  D_{2n} )&= \langle    \left(\begin{smallmatrix} 1 & & \\ & -1 & \\ & & -1 \end{smallmatrix}\right),
    \left(\begin{smallmatrix} \cos \alpha &  -\sin\alpha & \\ \sin\alpha & \cos \alpha & \\ & & 1 \end{smallmatrix}\right)
\rangle
\end{align*}
% %
% %
% % \quad \textrm{ and }\quad
% % (\rho\otimes\alpha) (  D_{2n} )= \langle    \left(\begin{smallmatrix} 1 & & \\ & -1 & \\ & & -1 \end{smallmatrix}\right),
% %     \left(\begin{smallmatrix} \cos \alpha &  -\sin\alpha & \\ \sin\alpha & \cos \alpha & \\ & & 1 \end{smallmatrix}\right)
% % \rangle
% % $$
and the dimension of the invariant spaces are
$$
 \dim(\mathbb R^3_{\rho})^{D_{2n}}= 1
\quad \textrm{ and }\quad
 \dim(\mathbb R^3_{\rho\otimes\alpha})^{D_{2n}}= 0.
$$
Summarizing:
% In both cases, when $\mathrm{Stab}(\tilde e)$ is not orientation preserving
\begin{equation}
 \label{eqn:dimsnonor}
 \dim  ( \mathbb R^3_\rho   )^{\mathrm{Stab}(\tilde e)}- \dim  ( \mathbb R^3_{\rho\otimes\alpha } )  ^{\mathrm{Stab}(\tilde e)}
 =\begin{cases}
0 & \textrm{ if }\mathrm{Stab}(\tilde e)\textrm{ is or.~pres.} \\
+1 & \textrm{ if }\mathrm{Stab}(\tilde e)\textrm{ is not o.~p.}
  \end{cases}
\end{equation}
Let $\mathsf{nops}_i(K)$ denote the number of $i$-cells of $K$ with non-orientation preserving stabilizer. Notice that all cells
with  non-orientation preserving stabilizer lie in the boundary of the underlying space of the orbifold $\partial|\mathcal O^2| $ (a subset of
 the boundary  of the orbifold $\partial\mathcal O^2 $),
that is a union of circles. By looking at the contribution of the full orbifolds
we may deduce:
\begin{equation}
 \label{eqn:fulls}
\mathsf{f}(\partial \mathcal O^2) =\mathsf{nops}_0(K)-\mathsf{nops}_1(K).
\end{equation}
Finally, using Equality~(a) and Equalities \eqref{eqn:twistedEuler}, \eqref{eqn:dimsnonor},
and \eqref{eqn:fulls}:
\begin{multline*}
 d_{\mathrm{oe}}-d_{\mathrm{tp}}=
 \dim H^1(\mathcal O^2,
 \mathbb R^n_{\rho\otimes \alpha})-
 \dim H^1(\mathcal O^2,\mathbb R^n_{\rho})
\\
 =\mathsf{nops}_0(K)-\mathsf{nops}_1(K)=
\mathsf{f}(\partial \mathcal O^2).
\end{multline*}
This finishes the proof of the lemma.
% The proof of (b) follows from   formulas \eqref{eqn:twistedEuler}, \eqref{eqn:dimsor}, \eqref{eqn:dimsnonor},
% and \eqref{eqn:fulls}.
\end{proof}

\begin{Remark}
The difference of dimensions also occurs when we consider both possible
 embedding of
 $\mathrm{Isom}(\mathbb H^2)$ in either $\mathrm{Isom}^+(\mathbb H^3)$
 (orientation embedding)
 or $\mathrm{Isom}(\mathbb H^3)$ (type preserving), according to whether a non-orientable
 isometry of $\mathbb H^3$ extends to the unique orientable or non-orientable isometry of $\mathbb H^3$.
For instance, consider a Fuchsian group $\Gamma<\mathrm{Isom}(\mathbb H^2)$
generated by a rotation of finite order and a reflection along a geodesic,
$\mathcal O^2_4$ in
 Figure~\ref{Figure:Exs}.
Its deformation space as Fuchsian group has dimension one,
the parameter being the distance between the fixed point of the rotation to the fixed line of the reflection,
both in $\mathbb H^2$:
\begin{enumerate}[(i)]
 \item The type preserving embedding $\Gamma_{\mathrm{tp}}$
is a group generated by a rotation and a reflection on a
plane in $\mathbb H^3$. The deformation space of $\Gamma_{\mathrm{tp}}$
has again dimension one, as the relative position between a plane and a line
in generic position is determined by their distance.
\item The orientation embedding $\Gamma_{\mathrm{oe}}$
is a group generated by two rotations along two axes, one of them of order two (the reflection in $\mathbb H^2$
extends to a rotation of order two in $\mathbb H^3$).
The relative position between two lines in $\mathbb H^3$
is determined by two parameters, hence the deformation space as quasi-Fuchsian group has dimension 2.
\end{enumerate}
Notice that in this case $\mathsf{f}(\partial \mathcal O^2) =1 $,
 it is Example~\ref{Example:On} below ($\mathcal O^2_4$ in
 Figure~\ref{Figure:Exs}).
\end{Remark}

\subsection{Examples}
\label{section:Examples}

% \subsection{2-orbifolds}

In this section we describe the deformation spaces for four
 2-orbifolds pictured in Figure~\ref{Figure:Exs}.
 They have small deformation spaces and the topology of the singularities is easy to describe.

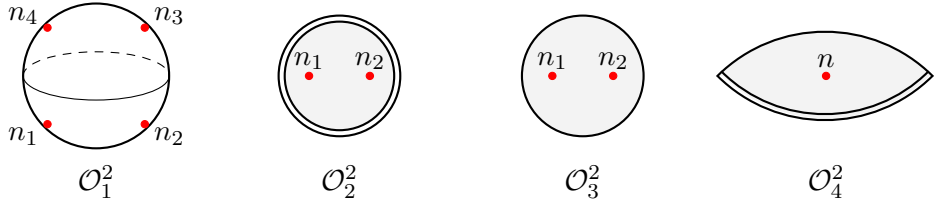
\begin{figure}[ht]
\begin{center}
 \begin{tikzpicture}[line join = round, line cap = round, scale=.8]
 \begin{scope}[shift={(0,0)}]
 \draw[thick] (0,1) circle(1.2);
     \draw[red, fill=red] (-.8,0.2) circle(0.06);
     \draw[red, fill=red] (.8,0.2) circle(0.06);
     \draw[red, fill=red] (.8,1.8) circle(0.06);
     \draw[red, fill=red] (-.8,1.8) circle(0.06);
      \draw (-1.2,1) arc[x radius = 1.2, y radius = .4, start angle =180, end angle= 360 ];
      \draw[dashed] (-1.2,1) arc[x radius = 1.2, y radius = .4, start angle =180, end angle= 0 ];
     \draw (-6/5,0) node{$n_1 $};
     \draw (6/5,0) node{$n_2 $};
     \draw (6/5,2) node{$n_3 $};
     \draw (-6/5,2) node{$n_4 $};
     \draw (0,-.75) node{$\mathcal O^2_1$};
   \end{scope}
       \begin{scope}[shift={(4,0)}]
        \fill[gray!30!white, opacity=.3] (0,1) circle(.9);
        \draw[thick] (0,1) circle(1);
         \draw[thick] (0,1) circle(.9);
      \draw[red, fill=red] (-.5,1) circle(0.06);
      \draw[red, fill=red] (.5,1) circle(0.06);
    \draw (-.5,1.25) node{$n_1 $};
     \draw (.5,1.25) node{$n_2 $};
     \draw (0,-.75) node{$\mathcal O^2_2$};
   \end{scope}
       \begin{scope}[shift={(8,0)}]
        \fill[gray!30!white, opacity=.3] (0,1) circle(1);
        \draw[thick] (0,1) circle(1);
      \draw[red, fill=red] (-.5,1) circle(0.06);
      \draw[red, fill=red] (.5,1) circle(0.06);
    \draw (-.5,1.25) node{$n_1 $};
     \draw (.5,1.25) node{$n_2 $};
     \draw (0,-.75) node{$\mathcal O^2_3$ };
   \end{scope}
       \begin{scope}[shift={(12,0)}]
\fill[gray!30!white, opacity=.3] (1.75,1)  arc [start angle=45, end angle=135, x radius=2.5cm, y radius =2.5cm] ;
\fill[gray!30!white, opacity=.3]  (1.7,1) arc [start angle=-45, end angle=-135, x radius=2.4cm, y radius =2.4cm] ;
\draw[thick] (1.75,1) arc [start angle=45, end angle=135, x radius=2.5cm, y radius =2.5cm] ;
\draw[thick] (1.75,1) arc [start angle=-45, end angle=-135, x radius=2.5cm, y radius =2.5cm] ;
 \draw[thick] (1.678,1.07) arc [start angle=-45, end angle=-135, x radius=2.4cm, y radius =2.4cm] ;
      \draw[red, fill=red] (0,1) circle(0.06);
%       \draw[red, fill=red] (.5,1) circle(0.06);
    \draw (0,1.25) node{$n $};
%      \draw (.5,1.25) node{$n_2 $};
     \draw (0,-.75) node{$\mathcal O^2_4$ };
   \end{scope}
\end{tikzpicture}
\end{center}
\caption{The four examples in Section~\ref{section:Examples}. The orbifold
$\mathcal O^2_1$ is the orientation covering of
$\mathcal O^2_2$, and
$\mathcal O^2_3$, of
$\mathcal O^2_4$ (for a suitable choice of labels).}
\label{Figure:Exs}
\end{figure}

 \begin{Example} Consider a sphere with 4 cone points
  $\mathcal O^2_1=S^2(n_1,n_2,n_3,n_4)$.
  Since $\mathcal O^2_1$ is hyperbolic, at
  least one of the $n_i$ is $\geq 2$.
  The dimension of the space of convex projective structures is
  $p=\dim  (   \mathrm{Proj}_{\mathrm{cvx}}(\mathcal O^2_1))= 8-2 k $, where
  $0\leq k\leq 3$ is the number of $n_i$ equal to $2$.
  The dimension of the Teichm\"uller space is
  $d=\dim (\mathrm{Teich}(\mathcal O^2_1))=2$.

  As $\mathcal O^2_1$
  is orientable and closed, $X(\mathcal O^2_1,\mathrm{SL}_4(\mathbb R))$
  has a neighborhood homeomorphic to
 $$
\mathbb R^{8-2k}\times  \mathrm{Cone}( \mathrm{UT}( S^{1}))=
\mathbb R^{8-2k}\times    \mathrm{Cone}( S^{1}  \sqcup S^{1})
   $$
Namely,    the neighborhood is homeomorphic to the union to 2 manifolds of dimension $10-2k$
that intersect  along a manifold of codimension 2.
 \end{Example}

   \begin{Example}
   Consider a disc with mirror boundary and two cone points of
 order $n_1$ and $n_2$, with at most one of the $n_i$ equal to 2. This is a non-orientable closed 2-orbifold,
 that we denote by  $\mathcal O^2_2=D^2(n_1,n_2;\emptyset)$. Its orientation covering
 is $S^2(n_1, n_1, n_2, n_2)$. Now the dimension of the space of
 convex projective structures is
   $p=\dim  (   \mathrm{Proj}_{\mathrm{cvx}}(\mathcal O^2_2))= 4-2 k $, where
  $0\leq k\leq 1$ is the number of $n_i$ equal to $2$. The dimension of the
  Teichm\"uller space is
  $d=\dim (\mathrm{Teich}(\mathcal O^2_2))=1$.

  As $\mathcal O^2_2$ is non-orientable, we distinguish the orientable embedding
  from the type preserving embedding:
  \begin{enumerate}[(i)]
   \item Consider the orientable embedding. A neighborhood in
   $X(\mathcal O^2_2, \mathrm{SL}_4(\mathbb R))$ is homeomorphic to
  $$
  \mathbb R^{4-2k}\times  \mathrm{Cone}( S^{0}\times S^0).
  $$
Here  $S^0\times S^0$ is the set with cardinality 4.
Namely,    the neighborhood is homeomorphic to the union to 2 manifolds of dimension $5-2k$ that intersect along a manifold of codimension one.
  \item Consider the type preserving embedding. A neighborhood  of the character  in
  $X(\mathcal O^2, \mathrm{SL}_{\pm}(4,\mathbb R))$ is homeomorphic to
  $$
  \mathbb R^{4-2k}\times  \mathrm{Cone}( S^{0}\times S^0)/\!\!\sim\,\, =
    \mathbb R^{4-2k}\times  \mathrm{Cone}(S^{0})\cong  \mathbb R^{5-2k}.
  $$
  Thus the neighborhood is topologically non-singular.

  \end{enumerate}
   \end{Example}

   \begin{Example} Next consider a disc with
   two cone points of order $n_1$ and $n_2$, with $(n_1,n_2)\neq (2,2)$,
        $\mathcal O^2_3=D^2(n_1,n_2)$. Even if the underlying space and the number of
        cone points is the same as the previous one,
        this orbifold is orientable and has boundary.
        The deformation space however looks like the previous example (orientable embedding):
        $\dim  (   \mathrm{Proj}_{\mathrm{cvx}}(\mathcal O^2_3))= 4-2 k $, where
  $0\leq k\leq 1$ is the number of $n_i$ equal to $2$ and
  $\dim (\mathrm{Teich}(\mathcal O^2_3))=1$.  As $\partial \mathcal O^2_3
  \neq\emptyset$,  $X(\mathcal O^2_3, \mathrm{SL}_4(\mathbb R))$ is again homeomorphic to
  $
  \mathbb R^{4-2k}\times  \mathrm{Cone}( S^{0}\times S^0)
   $.
   \end{Example}

   \begin{Example}
   \label{Example:On}
   Consider $\mathcal O^2_4$ a disc $D^2$ with one cone point of order $n\geq 3$
   and $\partial D^2$ is the union of a full orbifold and a mirror interval.
    This is a non-orientable orbifold with boundary, and its
    orientation cover is $D^2(n, n)$, the previous example.
    Now  $\dim  (   \mathrm{Proj}_{\mathrm{cvx}}(\mathcal O^2_4))= 2 $
     and
     $\dim (\mathrm{Teich}(\mathcal O^2_4))=1$.
       \begin{enumerate}[(i)]
   \item For the orientable embedding, a neighborhood in
   $X(\mathcal O^2_4, \mathrm{SL}_4(\mathbb R))$ is homeomorphic to
  $$
  \mathbb R^{2}\times  \mathrm{Cone}( S^{0}\times S^0)
  $$
  and the neighborhood is homeomorphic to the union of two 3-manifolds
  that intersect along a surface.
  \item Consider the type preserving embedding. As $\mathsf{f}(\partial \mathcal O^2_4)= 1$,
  $$\dim (\mathrm{Teich}(\mathcal O^2_4))-
  \mathsf{f}(\partial \mathcal O^2_4)=0.$$
  Thus all deformations of representations of $\pi_1(\mathcal O^2_4)$
  in $\mathrm{SL}^{\pm}_4(\mathbb R)$
  (after the type preserving embedding)
  remain in
  $\mathrm{SL}^{\pm}_3(\mathbb R)$.
    \end{enumerate}

   \end{Example}

\section{Projective deformations of hyperbolic 3-orbifolds}
\label{Section:projdef3}

A hyperbolic 3-orbifold  has a natural
projective structure, by using the projective model of hyperbolic geometry,
and there is no known general criterion to determine
the deformation space of those projective structures,
cf.~\cite{CLT1, CLT2, HeusenerPorti11}. 
Motivated by this question, the following definition 
was introduced in a joint paper with Heusener \cite{HeusenerPorti11}.

%
% We introduce first the notion of projectively infinitesimally rigidity
% with respect to the boundary. For finite
% volume manifolds this implies smoothness, see
% \cite{BDL}. Here we consider orbifolds of finite type in general.

\begin{Definition}
\label{Definition:infinitesimal}
A compact hyperbolic 3-orbifold $\mathcal O^3$ is projectively infinitesimally rigid
with respect to the boundary if
the inclusion induces an injection
$$
0\to H^1(\mathcal O^3,\mathfrak{sl}_4(
\mathbb R)) \to
H^1(\partial \mathcal O^3,\mathfrak{sl}_4(
\mathbb R)).
$$
\end{Definition}

%For finite volume hyperbolic manifolds, infinitesimal rigidity
%with respect to the boundary  was defined  in a joint paper with Heusener
%\cite{HeusenerPorti11}.
For finite volume hyperbolic 3-manifolds, Ballas,  Danciger,  and Lee
\cite[Theorem~3.2]{BDL}
proved that infinitesimal rigidity with respect to the boundary implies that the space of
$\mathrm{PSL}_4(\mathbb R)$ characters is smooth.
Here we consider compact 3-orbifolds with hyperbolic interior, possibly of infinite volume, namely orientable 3-orbifolds of finite type.

\begin{Example}
 A product $\mathcal O^3=\mathcal O^2\times [0,1]$ it is infinitesimally
 projectively rigid for any hyperbolic structure.
\end{Example}

\begin{Definition}
We say that one end of a hyperbolic 3-orbifold is \emph{Fuchsian} if its holonomy is a Fuchsian group: namely the end corresponding to a totally geodesic
boundary component.
\end{Definition}

In this definition, finite volume ends (e.g.~corresponding to
a Euclidean boundary component) are not considered Fuchsian.

This section is devoted to the proof of
Theorem~\ref{Theorem:3dimIntro} in the introduction, that we restate for convenience:

\begin{Theorem}[Theorem~\ref{Theorem:3dimIntro}]
\label{Theorem:3dim}
Let $\mathcal O^3$ be a compact orientable orbifold, with $\partial \mathcal O^3\neq \emptyset$ and hyperbolic interior
$\mathrm{int}( \mathcal O^3 )$, so  that it is not elementary, nor Fuchsian.
Assume that $\mathcal O^3$ is infinitesimally projectively rigid
with respect to the boundary.
Then
 the character of the
hyperbolic holonomy is a smooth point of
$X(\mathcal O^3, \mathrm{SL}_4(\mathbb R))$
if,  and only if,
all ends of
$\mathrm{int}( \mathcal O^3 )$ are either non-Fuchsian or turnovers. Furthermore, the singularity is quadratic.
 \end{Theorem}

The variety of characters of
$X(\mathcal O^3, \mathrm{SL}_4(\mathbb R))$ 
has a natural analytic structure
at the character of the holonomy,
but a priori it may be non-reduced, it may contain points or subvarieties with multiplicity higher than one. We follow the following convention:

\begin{Remark}
A point with higher multiplicity is considered a singular point. We use the convention that a point is singular
when the dimension of the Zariski tangent space is higher than the dimension of the (any) component containing it.
\end{Remark}

In particular, a point that lies in more than one irreducible component is singular.

\begin{Example} Consider again a product $\mathcal O^3=\mathcal O^2\times [0,1]$.
When $\mathcal O^2$ is Fuchsian,
its holonomy is contained in
$\mathrm{SO}(2,1)\subset\mathrm{SL}_3(\mathbb R)$, hence reducible as representation in
$\mathrm{SL}_4(\mathbb R)$.  By Theorem~\ref{Theorem:orientable},  the space of characters in
 $\mathrm{SL}_4(\mathbb R)$ is singular. But when  $ \mathcal O^2  $  is
 non-Fuchsian it is smooth (by Lemma~\ref{Lemma:03irr} below
 its  holonomy in
$\mathrm{SL}_4(\mathbb R)$ is $\mathbb{C}$-irreducible and
Proposition~\ref{Proposition:Goldman} applies).
\end{Example}

The proof of Theorem~\ref{Theorem:3dim} is divided into three statements:
Propositions~\ref{Proposition:Smoothness}, \ref{Proposition:FuchsianSingular}, and~\ref{Proposition:Quadratic}.

Proposition~\ref{Proposition:Smoothness}
shows the smoothness of the variety of characters when there are
no Fuchsian ends other than a turnover. It  is  proved in two subsections.
In Subsection~\ref{Subsection:hol} we establish some preliminary results on the holonomy of a hyperbolic 3-orbifold,
and in Subsection~\ref{subsection:Smoothness} we complete the proof,
relying on naturality of Goldman's obstructions to integrability and
smoothness of the character varieties of the boundary components.

In Proposition~\ref{Proposition:FuchsianSingular} we show that,
when there is a Fuchsian end other than a turnover, the variety
of characters is singular. This is also done in two subsections.
The first step consists in deforming the holonomy representation
to another representation
in $\mathrm{Isom}^+(\mathbb H^3)$ with no Fuchsian ends. This is done in
Subsection~\ref{Subsection:Deforming}
and then in Subsection~\ref{Subsection:ZT}
we show that the dimension of the Zariski tangent space strictly
decreases under this deformation, establishing the singularity.

Finally in Subsection~\ref{Section:Quadratic} we establish
that the singularity is quadratic, relying again on the boundary,
where the fact that it is quadratic  is due to Goldman
and Millson \cite{GoldmanMillson}.

 \subsection{The projective holonomy of a hyperbolic 3-orbifold}
\label{Subsection:hol}
 In this subsection we discuss some preliminary results on the hyperbolic holonomy
 representation of a three orbifold
 $\pi_1(\mathcal O^3)\to \mathrm{SO}_0(3,1)$ composed with
 the inclusion
 $\mathrm{SO}_0(3,1)\subset
 \mathrm{SL}_4(\mathbb R)$.

% For the proof of the theorem we need the following:

\begin{Lemma}
\label{Lemma:03irr}
 Let $\mathcal O^3$ be a compact orientable orbifold, with  hyperbolic interior. If it is not elementary nor
Fuchsian, then its holonomy representation
 $\pi_1(\mathcal O^3)\to \mathrm{SO}_0(3,1)\subset
 \mathrm{SL}_4(\mathbb R)$ is $\mathbb C$-irreducible.
\end{Lemma}

\begin{proof}
% [Proof of Lemma~\ref{Lemma:03irr}]
The proof is by contradiction.
 Assume there is a proper subspace  $V\subset \mathbb C^4$ invariant by the holonomy representation, then the image of this holonomy
 is contained in the  closed subgroup $H=\{g\in \mathrm{SO}_0(3,1)\mid g V= V\}$, that has finitely many components.
 We may assume that the identity component  $H_0$ of $H$ is nontrivial,
 otherwise  $H$ would be finite and hence $\mathcal O^3$ elementary, yielding a contradiction.
Since $\mathbb C^4$ is irreducible as  $\mathrm{SO}_0(3,1)$-module,
  $H_0$
 is a proper connected subgroup of $\mathrm{SO}_0(3,1)$.
  Consider the action of this subgroup $H_0$ in hyperbolic space:
  as it is a proper subgroup, $H_0$  preserves either a totally geodesic plane,
a line,  or a point in either hyperbolic space or the ideal boundary. The whole subgroup $H$ also preserves the same
totally geodesic plane, line, or (finite or ideal) point, because $H_0$ is a finite index normal subgroup of $H$.
 When $H$ preserves a totally geodesic plane,
 $\mathcal O^3$ is Fuchsian, and in all other cases $\mathcal O^3$ is elementary. This contradiction concludes the proof.
\end{proof}

 Lemma~\ref{Lemma:03irr} applies not only to the orbifold
$\mathcal O^3$ of Theorem~\ref{Theorem:3dim} but also to its ends;
namely to  the orbifold $\mathbb H^3/\Gamma$, where
$\Gamma<\pi_1(\mathcal O^3)$ denotes the peripheral 
subgroup corresponding to an end of $\mathcal O^3$.

\begin{Corollary}
\label{Corollary:03irr}
  Let $\mathcal O^3$ be a compact orientable orbifold, with  hyperbolic interior. If it is not elementary nor
Fuchsian, then the invariant subspace of the Lie algebra is
trivial:
$$
\mathfrak{sl}_4(\mathbb R)^{ \pi_1(\mathcal O^3) }=0.
$$
In particular  $H^0(\mathcal O^3 , \mathfrak{sl}_4(\mathbb R))=0$.
\end{Corollary}

\begin{proof}
 By contradiction, assume that there exists $0\neq
v\in\mathfrak{sl}_4(\mathbb R)^{ \pi_1(\mathcal O^3) }$.
Then the holonomy of $\pi_1(\mathcal O^3)$ commutes with the one
parameter real group $\{\exp(t v)\mid t\in\mathbb R\}$.
By Lie-Kolchin theorem, this real group has an invariant complex
line in
$\mathbb C^4$. This line is also preserved by
$\pi_1(\mathcal O^3)$, contradicting Lemma~\ref{Lemma:03irr}.
The last assertion of the corollary follows from the natural isomorphism
$H^0(\mathcal O^3 , \mathfrak{sl}_4(\mathbb R))
\cong
H^0(\pi_1(\mathcal O^3) , \mathfrak{sl}_4(\mathbb R))
\cong
\mathfrak{sl}_4(\mathbb R)^{ \pi_1(\mathcal O^3) }$.
\end{proof}

\begin{Remark}
 \label{Remark:02irr}
 As $\mathbb{C}^{n+1}$ is $\mathrm{SO}(n,1)$-irreducible,
 the very same argument proves
 the analog statement in any dimension.
 In particular,  for a non-elementary hyperbolic
 2-orbifold $\mathcal O^2$,   its (Fuchsian) holonomy representation
 $\pi_1(\mathcal O^2)\to \mathrm{SO}_0(3,1)\subset
 \mathrm{SL}_3(\mathbb R)$ is $\mathbb C$-irreducible
 and $H^0(\mathcal O^2 , \mathfrak{sl}_3(\mathbb R))=0$.
\end{Remark}

% To prepare the proof Theorem~\ref{Theorem:3dim}
Next we require
some more preliminary results
on cohomology with twisted coefficients in the Lie algebra
$\mathfrak{sl}_4(
\mathbb R)$.
For an orbifold $\mathcal O^3$ satisfying the hypothesis of
Theorem~\ref{Theorem:3dim},
let
$$
\partial \mathcal O^3= \partial_1\mathcal O^3_1\cup \cdots
\cup  \partial_k\mathcal O^3
$$
denote the decomposition in
connected components, so that
$$
H^*(\partial \mathcal O^3,\mathfrak{sl}_4(
\mathbb R)
)=\bigoplus_{i=1}^k H^*( \partial_i  \mathcal O^3,\mathfrak{sl}_4(
\mathbb R)) \cong
\bigoplus_{i=1}^k H^*(\pi_1( \partial_i  \mathcal O^3),
\mathfrak{sl}_4(
\mathbb R)).
$$
Being infinitesimally projectively rigid means that the restriction induces an inclusion
\begin{equation}
 \label{eqn:inclusion}
 H^1(\mathcal O^3,\mathfrak{sl}_4(
\mathbb R)
)\hookrightarrow H^1(\partial\mathcal O^3,\mathfrak{sl}_4(
\mathbb R)).
\end{equation}
We furthermore have:

\begin{Lemma}
\label{Lemma:resiso}
Let $\mathcal O^3$ be a compact orientable orbifold, with $\partial \mathcal O^3\neq \emptyset$ and with  hyperbolic interior
$\mathrm{int}( \mathcal O^3 )$, that is  neither elementary and nor
a product
$\mathcal O^2\times [0,1]$.
Assume that $\mathcal O^3$ is infinitesimally projectively rigid
with respect to the boundary. Then:
\begin{enumerate}[(i)]
 \item $\dim H^1(\mathcal O^3,\mathfrak{sl}_4(
\mathbb R))
=\frac12 \dim H^1(\partial\mathcal O^3,\mathfrak{sl}_4(
\mathbb R))$.
 \item The restriction induces an isomorphism
\begin{equation}
\label{eqn:periph}
H^2(\mathcal O^3,\mathfrak{sl}_4(
\mathbb R))\cong H^2(\partial\mathcal O^3,\mathfrak{sl}_4(
\mathbb R)).
 \end{equation}
\end{enumerate}
\end{Lemma}

\begin{proof}[Proof of Lemma~\ref{Lemma:resiso}]
By~\eqref{eqn:inclusion}, the long exact sequence
in cohomology of the pair
$(\mathcal O^3,\partial\mathcal O^3)$
writes as
 \begin{equation*}
%   \label{eqn:epimorphism}
   0\to H^1(\mathcal O^3,\mathfrak{sl}_4
(\mathbb R))\overset{i^*} \to
H^1(\partial \mathcal O^3,\mathfrak{sl}_4 (\mathbb R))
\overset\delta\to
H^2(\mathcal O^3,\partial \mathcal O^3,\mathfrak{sl}_4 (\mathbb R))
\to\dotsb
 \end{equation*}
 We claim that
the maps $i^*$ and $\delta$ are compatible with the pairings
used in Theorem~\ref{TheoremPD} induced by the Killing form.
Namely, let
$$
B\colon \mathfrak{sl}_4 (\mathbb R)  \times
\mathfrak{sl}_4 (\mathbb R)\to\mathbb R
$$
denote the Killing form. We have the non-degenerate pairings
$$
B(\cdot \cup\cdot)\colon H^1(\mathcal O^3, \mathfrak{sl}_4 (\mathbb R))
\times
H^2(\mathcal O^3,\partial \mathcal O^3, \mathfrak{sl}_4 (\mathbb R))
\to
H^3(\mathcal O^3,\partial \mathcal O^3,\mathbb R)
\cong \mathbb R,
$$
and, for $j=1,\dotsc,k$,
$$
(\cdot,\cdot)_j=
B(\cdot \cup\cdot)\colon H^1(\partial_j \mathcal O^3, \mathfrak{sl}_4 (\mathbb R))
\times
H^1(\partial_j \mathcal O^3, \mathfrak{sl}_4 (\mathbb R))
\to
H^2(\partial_j\mathcal O^3,\mathbb R)
\cong \mathbb R.
$$
We denote their (orthogonal) sum  by
$$
\Psi=\sum(\cdot,\cdot)_j
\colon H^1(\partial \mathcal O^3, \mathfrak{sl}_4 (\mathbb R))
\times
H^1(\partial \mathcal O^3, \mathfrak{sl}_4 (\mathbb R))
\to \mathbb R.
$$
By naturality, for $a\in  H^1(\mathcal O^3, \mathfrak{sl}_4 (\mathbb R))$ and $b\in H^1(\partial \mathcal O^3, \mathfrak{sl}_4 (\mathbb R))$,
\begin{equation}
 \label{eqn:compatibility}
\Psi(i^*(a),b)= B(a\cup \delta(b))).
\end{equation}

Next we claim that the image of $i^*$ is a maximal isotropic subspace
of~$\Psi$. This is a standard argument done for instance
in \cite{HeusenerPorti11, PortiDim,  Sikora}, but we provide it by completeness. For every
$a\in  H^1(\mathcal O^3, \mathfrak{sl}_4 (\mathbb R))$, the compatibility condition \eqref{eqn:compatibility} and exactness yield
$
\Psi(i^*(a),i^*(a))= B(a\cup \delta(i^*(a))))=0.
$
Thus the image of $i^*$ is isotropic. To get maximality, if $b\in H^1(\partial \mathcal O^3, \mathfrak{sl}_4 (\mathbb R))$
satisfies $\Psi(i^*(a),b)$ for every
$a\in  H^1(\mathcal O^3, \mathfrak{sl}_4 (\mathbb R))$, then by
\eqref{eqn:compatibility}
and non-degeneracy of the pairing, $\delta(b)=0$.
Therefore $b$ belongs to the image of $i^*$.
The pairing $\Psi$ being skew-symmetric and non-degenerate, its
maximal isotropic subspaces are half-dimensional. This proves the Assertion (i) of the lemma.

For the second assertion,
the maps of the long exact sequence
\begin{equation}
 \label{eqn:H1}
H^1(\mathcal O^3,\partial \mathcal O^3,\mathfrak{sl}_4(\mathbb R))\to H^1(\mathcal O^3,\mathfrak{sl}_4(\mathbb R))
\end{equation}
and
\begin{equation}
 \label{eqn:H2}
 H^2(\mathcal O^3,\partial \mathcal O^3,\mathfrak{sl}_4(\mathbb R))\to H^2(\mathcal O^3,\mathfrak{sl}_4(\mathbb R))
\end{equation}
are also compatible with the pairing $B(\cdot\cup\cdot)$, similarly to \eqref{eqn:compatibility}.
The map \eqref{eqn:H1} is trivial, because the next map in the long exact sequence
of the pair,
$i^*$, is injective.
Then, using compatibility with the pairing and non-degeneracy, one can prove that
the map \eqref{eqn:H2} is also trivial.
Next
we consider the continuation of the long exact sequence of
the pair:
$$
0\to H^2(\mathcal O^3,\mathfrak{sl}_4(\mathbb R))\to
 H^2(\partial\mathcal O^3,\mathfrak{sl}_4(\mathbb R))
 \to
 H^3(\mathcal O^3,\partial\mathcal O^3\mathfrak{sl}_4(\mathbb R)).
$$
The conclusion then comes from the vanishing of
$H^3(\mathcal O^3,\partial\mathcal O^3,\mathfrak{sl}_4(\mathbb R))$:
it is dual to $H^0(\mathcal O^3;\mathfrak{sl}_4(\mathbb R))$
that  vanishes
by Corollary~\ref{Corollary:03irr}.
\end{proof}

\subsection{Smoothness of varieties of representations}
\label{subsection:Smoothness}

In this subsection we prove one of the
implications of Theorem~\ref{Theorem:3dim},
relying on the smoothness of varieties of representations on
the components  of $\partial \mathcal O^3$.
As in the previous subsection,
let
$$
\partial \mathcal O^3= \partial_1\mathcal O^3_1\cup \cdots
\cup  \partial_k\mathcal O^3
$$
denote the decomposition into connected components of the
boundary.

Assume that $\mathcal O^3$ is hyperbolic and let
$\rho_i\colon \pi_1(\partial_i \mathcal O^3)\to \mathrm{SL}_4(\mathbb R)$
denote the restriction of its holonomy, composed with the
inclusion $\mathrm{SO}(3,1)\subset\mathrm{SL}_4(\mathbb R)$.

\begin{Lemma}
\label{Lemma:rho_iNF}
If either $\rho_i$ is non-Fuchsian or $\partial_i \mathcal O^3$
is a turnover,
then
$\hom(\pi_1(\partial_i\mathcal O^3),\mathrm{SL}_4(\mathbb R))$
is smooth at
$\rho_i$. Equivalently, Goldman's obstructions to integrability
of Lemma~\ref{Lemma:GoldmanObstructions}
vanish.
 \end{Lemma}

% We point out that Goldman's obstructions vanish, as we shall use them later.

\begin{proof}
We distinguish three cases, according to the topology of the 2-orbifold
$\partial_i\mathcal O^3$.

First assume that  $\partial_i \mathcal O^3$ is Euclidean
(ie the corresponding end has finite volume).
If $\partial_i \mathcal O^3$ is a
manifold, then $\partial_i \mathcal O^3$ is a 2-torus and in this case \cite[Lemma~3.4]{BDL} Ballas, Danciger, and Lee prove that for a peripheral torus
the space of representations in $\mathrm{SL}_4(\mathbb R)$ is smooth,
and all obstructions to integrability
in Lemma~\ref{Lemma:GoldmanObstructions}
vanish, see also
\cite[Proof of Theorem~3.2]{BDL}. If the Euclidean  orbifold
$\partial_i \mathcal O^3$ is not a manifold, then by Bieberbach theorem
$\partial_i \mathcal O^3$ has a finite regular
covering that is a 2-torus $T^2$, and the cohomology
of $\partial_i \mathcal O^3$
is equivalent to the equivariant cohomology of $T^2$.
So the claim follows  from naturality of the obstruction by using
equivariance.

Next assume that
$\partial_i \mathcal O^3$ is hyperbolic and $\rho_i$ is not Fuchsian.
In this case Proposition~\ref{Proposition:Goldman} applies to $\rho_i$,
because by Lemma~\ref{Lemma:03irr}
 $\rho_i$ is $\mathbb C$-irreducible. As we are interested in
the obstruction to integrability, it is worth recalling the proof
of Proposition~\ref{Proposition:Goldman} in \cite{Goldman}: by Corollary~\ref{Corollary:03irr}
$H^0(\pi_1(\partial_i \mathcal O^3),\mathfrak{sl}_4(\mathbb R))=0$
and therefore, by Poincaré duality,
$H^2(\pi_1(\partial_i \mathcal O^3), \mathfrak{sl}_4(\mathbb R) )=0$.
Thus there are no obstructions to integrability   at all.

Finally assume that
$\partial_i \mathcal O^3$
is a turnover,
we have shown that every representation of $\pi_1(\partial_i \mathcal O^3)$ in $\mathrm{SL}_4(\mathbb R)$
is conjugate to a representation in $\mathrm{SL}_3(\mathbb R)$ by Corollary~\ref{Corollary:turnover}.
It follows that the space of representations in $\mathrm{SL}_4(\mathbb R)$ is smooth, as
the space of representations in $\mathrm{SL}_3(\mathbb R)$ is smooth by \cite{Goldman} (it is the Choi-Goldman space).
To describe explicitly Goldman's obstructions, we consider the decomposition of the Lie algebra by $\pi_1(\partial_i \mathcal O^3)$-modules
of \eqref{eqn:directsum}:
$$
\mathfrak{sl}_4(\mathbb R)=
\mathfrak{sl}_3(\mathbb R) \oplus \mathbb R^3_{\rho_i}
\oplus \mathbb R^3_{\rho_i^3}\oplus
\mathbf d,
$$
where $\mathbf d\cong\mathbb R$ is the subspace of diagonal
matrices in
$ \mathfrak{sl}_4(\mathbb R) $ that commute with every matrix in
$ \mathfrak{sl}_3(\mathbb R )$. The vanishing of the obstructions follows from the fact that $H^1(\partial_i \mathcal O^3, \mathbb R^3_{\rho_i})=
H^1(\partial_i \mathcal O^3, \mathbb R^3_{\rho_i^*})=
H^1(\partial_i \mathcal O^3, \mathbf d)= 0$ (see Corollary~\ref{Corollary:turnover} and the previous paragraph), and
$H^2( \partial_i \mathcal O^3, \mathfrak{sl}_3(\mathbb R) )=0$, by
Remark~\ref{Remark:02irr}.
\end{proof}

% %

The following is one of the implications of Theorem~\ref{Theorem:3dim}.

\begin{Proposition}
\label{Proposition:Smoothness}
 Under the hypothesis of  Theorem~\ref{Theorem:3dim},
 if all ends of
$\mathrm{int}( \mathcal O^3 )$ are either non-Fuchsian or turnovers,
then the character of $\rho$ is a smooth point
of $X(\mathcal O^3, \mathrm{SL}_4(\mathbb R))$.
\end{Proposition}

\begin{proof} Since Goldman's obstructions to integrability
are natural, by Lemmas~\ref{Lemma:resiso}
and
\ref{Lemma:rho_iNF}
all obstructions to integrability in
$H^2(\mathcal O^3,\mathfrak{sl}_4(\mathbb{R}))$ vanish.
So  we have smoothness by applying a theorem of Artin.
More precisely, the vanishing of obstructions
yields a formal deformation of $\rho$ in the direction of any cocycle $z$
in the Zariski tangent space, where
formal means a power series perhaps not convergent. Then the theorem of Artin asserts that
the formal power series can be replaced by a converging series with the same tangent vector
\cite{Artin}. As every vector in the Zariski tangent space is tangent to a path, it is a smooth point.
 See \cite{GoldmanMillson,HPS}
for more details.
 \end{proof}
%
%
%
% ***************************************************************************
%
%

\subsection{Deforming
representations in $\mathrm{Isom}^+(\mathbb{H}^3)$.}
\label{Subsection:Deforming}

% In this subsection we prove the other inplication  of Theorem~\ref{Theorem:3dim}.
% %
%  Let $\mathcal O^3$ be as in the theorem.  Assuming that
% $\mathrm{int}(\mathcal O^3)$ has a Fuchsian end other than a turnover, we
% shall show that the hyperbolic structure can be defromed so that  none of the ends
% (other than turnovers) is Fuchsian, in Lemma~\ref{Proposition:NFD} below. Next we shall show that the dimension
% of the Zariski
% tangent space depends on the number of Fuchsian ends in Lemma~\ref{Lemma:dimZTF}, so that
%  the dimension of
% its Zariski tangent space at the initial structure is strictly larger than the dimension
% at a character nearby. This implies that this character is a singular point.

To continue with the proof of Theorem~\ref{Theorem:3dim}, in the next
proposition we prove that the holonomy can be perturbed to a
representation with similar algebraic properties,
such that the image of peripheral subgroups other than turnovers
are not conjugate to subgroups of $\mathrm{Isom}^+(\mathbb H^2)$
(eg no Fuchsian ends other than turnovers).
It is not clear than this perturbed
representation is the holonomy of a hyperbolic
structure, this would require different techniques and we do not use
hyperbolicity. All we need is that the dimension of the Zariski tangent
space decreases, so that the initial character is a singular point.

\begin{Proposition}
\label{Proposition:NFD}
Let  $\mathcal O^3$ be a compact orientable 3-orbifold with hyperbolic interior and holonomy $\rho_0\in\hom(\pi_1(\mathcal O^3,\mathrm{Isom}^+(\mathbb{H}^3))$.
There is a representation $\rho\in \hom(\pi_1(\mathcal O^3),
\mathrm{Isom}^+(\mathbb{H}^3))$
arbitrarily close to $\rho_0$ such that, for every
component $\partial_i\mathcal O^3$ of $\partial\mathcal O^3$ other than a turnover,
$\rho(\pi_1( \partial_i\mathcal O^3 )) $ is not conjugated to
a subgroup of
$\mathrm{Isom}^+(\mathbb{H}^2)$.
%
% a deformation $\rho'$ of the hyperbolic holonomy representation of $\mathcal O^3$
% (eg a deformation in $\hom(\pi_1(\mathcal O^3),\mathrm{SO}(3,1))$)
%  such that none of the ends other than a turnover is Fuchsian.
 \end{Proposition}

\begin{proof}
We use the group isomorphisms
$$
\mathrm{Isom}^+(\mathbb{H}^3)\cong
\mathrm{SO}_0(3,1)\cong \mathrm{PSL}_2(\mathbb C)\cong
\mathrm{SO}(3,\mathbb C).
$$
The character $\chi_0$ of the holonomy~$\rho_0$ is a smooth point of
the variety of characters $X(\mathcal O^3, \mathrm{PSL}_2(\mathbb C) )$
with tangent space $H^1(\mathcal O^3,\mathfrak{sl}_2(\mathbb C) )$, as shown by M.~Kapovich in
\cite[Section~8.8]{KapovichBook}. Furthermore, Kapovich proves that the restriction
$$\mathrm{res}\colon
X(\mathcal O^3, \mathrm{PSL}_2(\mathbb C) )\to X(\partial \mathcal O^3, \mathrm{PSL}_2(\mathbb C) )=\prod_i X(\partial_i \mathcal O^3, \mathrm{PSL}_2(\mathbb C) )
$$
is an immersion in a neighborhood of $\chi_0$, with (injective) tangent map
$$
\mathrm{res}_*\colon
 H^1(\mathcal O^3,\mathfrak{sl}_2(\mathbb C) )
\to
H^1(\partial \mathcal O^3,\mathfrak{sl}_2(\mathbb C) )
\cong \bigoplus_i H^1(\partial_i \mathcal O^3,\mathfrak{sl}_2(\mathbb C) ),
$$
the natural map
induced by the inclusion $\partial \mathcal{O}^3\subset \mathcal{O}^3$.

For each hyperbolic component $\partial_i\mathcal O^3$ other than a turnover,
its Teichm\"uller space (and hence
$X(\partial_i \mathcal O^3, \mathrm{PSL}_2(\mathbb C) )$) has positive dimension.
We claim the following:

\begin{Lemma}
\label{Lemma:resi}
 There exists a tangent vector $v\in  H^1(\mathcal O^3,\mathfrak{sl}_2(\mathbb C) )$
 such that for each component $\partial_i \mathcal O^3$ other than a turnover
 $$(\mathrm{res}_i)_*(v)\neq 0,
 $$
 where $(\mathrm{res}_i)_*\colon
 H^1(\mathcal O^3,\mathfrak{sl}_2(\mathbb C) )
\to
H^1(\partial_i \mathcal O^3,\mathfrak{sl}_2(\mathbb C) )$ is the map induced
by the restriction to that component.
\end{Lemma}

%
%
% that the image of the restriction to this component
% \begin{equation}
% (\mathrm{res}_i)_*\colon
%  H^1(\mathcal O^3,\mathfrak{sl}_2(\mathbb C) )
% \to
% H^1(\partial_i \mathcal O^3,\mathfrak{sl}_2(\mathbb C) ).
%  \label{eqn:resi}
% \end{equation}
% is non zero,

\begin{proof}[Proof of Lemma~\ref{Lemma:resi}]
We consider the same bilinear form as in Theorem~\ref{TheoremPD} and the
proof of Lemma~\ref{Lemma:resiso}, but with the  complex valued Killing form
$$
\mathfrak{sl}_2(\mathbb C) \times \mathfrak{sl}_2(\mathbb C)
\to\mathbb C.
$$
Namely, for $j=1,\dotsc, k$ we denote by
$$
(\cdot,\cdot)_j\colon H^1( \partial_j \mathcal O^3,
\mathfrak{sl}_2(\mathbb C))\times
H^1( \partial_j \mathcal O^3,
\mathfrak{sl}_2(\mathbb C))\to\mathbb C
$$
the skew-symmetric bilinear form obtained by composing the cup product
and the $\mathbb C$-Killing form. Its direct sum is denoted by
$$
\Psi=\sum (\cdot,\cdot)_j\colon H^1( \partial \mathcal O^3,
\mathfrak{sl}_2(\mathbb C))\times
H^1( \partial \mathcal O^3,
\mathfrak{sl}_2(\mathbb C))\to\mathbb C.
$$
The very same argument as in  the proof of Lemma~\ref{Lemma:resiso}
yields that
the image of $\mathrm{res}_*$
is a maximal isotropic
subspace of $\Psi$.
Since $\Psi$ is a component-wise sum of non-degenerate
forms, this implies that
 the image of each morphism
$(\mathrm{res}_i)_*\colon
 H^1(\mathcal O^3,\mathfrak{sl}_2(\mathbb C) )
\to
H^1(\partial_i \mathcal O^3,\mathfrak{sl}_2(\mathbb C) )$
% $(\mathrm{res}_i)_*$
% in \eqref{eqn:resi}
is nonzero, provided that
$H^1(\partial_i \mathcal O^3,\mathfrak{sl}_2(\mathbb C) )\neq 0$, eg
provided that $\partial_i \mathcal O^3$ is not a turnover.
The last assertion is equivalent to the fact that
$\ker (\mathrm{res}_i)_*$ is not the whole space  $H^1(\mathcal O^3,
\mathfrak{sl}_2(\mathbb C) )$  (whenever $\partial_i \mathcal O^3$ is not a
turnover). Thus the union of those kernels is not the whole
$H^1(\mathcal O^3,
\mathfrak{sl}_2(\mathbb C) )$  and the lemma follows.
% % there exists
% %  $v\in
% %  H^1(\mathcal O^3,\mathfrak{sl}_2(\mathbb C) )$ be a tangent vector such
% %  that
% %  $(\mathrm{res}_i)_*(v)\neq 0$ for each $i$ such that $\partial_i \mathcal
% %  O$
% %  is hyperbolic and not a turnover.
\end{proof}

We continue the proof of Proposition~\ref{Proposition:NFD}. Assume the component
$\partial_i \mathcal O^3$ is Fuchsian, namely
 $\rho_0(\pi_1(\partial_i \mathcal O^3))$
 preserves a totally geodesic $\mathbb H^2\subset \mathbb H^3$.  Then after conjugation
 it is contained in $\mathrm{PSL}(2,\mathbb R)$, and we have a decomposition
 $$
 H^1(\partial_i\mathcal{O}^3, \mathfrak{sl}_2(\mathbb C) )=
 H^1(\partial_i\mathcal{O}^3, \mathfrak{sl}_2(\mathbb R) )
 \oplus
\sqrt{-1}\, H^1(\partial_i\mathcal{O}^3, \mathfrak{sl}_2(\mathbb R) ).
 $$
 After multiplying by a generic complex number, we may assume that,
 for every Fuchsian component $\partial_i\mathcal O^3$
different from a turnover, the vector
 $v\in  H^1(\mathcal O^3,\mathfrak{sl}_2(\mathbb C) )$
of Lemma~\ref{Lemma:resi} satisfies
 \begin{equation}
  \label{eqn:resv}
   (\mathrm{res}_i)_*(v) \not\in  H^1(\partial_i\mathcal{O}^3, \mathfrak{sl}_2(\mathbb R) ).
 \end{equation}
We consider
 a deformation $\{\rho_t\}_{t\in [0,\varepsilon)}$ of the holonomy whose
 path of characters
$\{\chi_{\rho_t}\}_{t\in [0,\varepsilon)}$ in
$X( \mathcal O^3,\mathrm{PSL}_2(\mathbb C))$ is tangent to~$v$.

We claim that
for $t>0$ sufficiently small $\rho_t$
satisfies the conclusion of the proposition. Not being Fuchsian is an open property, so we just need to deal
with components $\partial_i\mathcal O^3$ different from a turnover that are Fuchsian, namely
 ${\rho_0}(\pi_1(\partial_i\mathcal O^3))$
is conjugate to a subgroup of $\mathrm{PSL}_2(\mathbb R)$.
If $\partial_i\mathcal O^3$ is Fuchsian, but not a turnover,
% then  $\chi_{\rho_i}=
% \mathrm{res}_i(\chi_\rho)\in X( \mathcal O^3,\mathrm{PSL}_2(\mathbb R))$.
then condition \eqref{eqn:resv} means that the path $\rho_t$  restricted to
$\pi_1(\partial_i\mathcal O^3)$ is transverse to the orbit by conjugation
of $\hom (\pi_1(\partial_i \mathcal O^3), \mathrm{PSL}_2(\mathbb R))$ inside
$\hom (\pi_1(\partial_i \mathcal O^3), \mathrm{PSL}_2(\mathbb C))$.
Thus
$\rho_t$  restricted to
$\pi_1(\partial_i\mathcal O^3)$ does not preserve a totally geodesic
hyperbolic plane for $t>0$.
\end{proof}

\subsection{Computing Zariski tangent spaces}
\label{Subsection:ZT}

In Proposition~\ref{Proposition:FuchsianSingular} we prove one more piece of Theorem~\ref{Theorem:3dim}.
The argument is based on
a computation of Zariski tangent spaces (Lemma~\ref{Lemma:dimZTF} below) and
Proposition~\ref{Proposition:NFD} in the last subsection.

\begin{Lemma}
\label{Lemma:dimZTF}
Let $\mathcal O^3$ be a 3-orbifold as in the hypothesis of
Theorem~\ref{Theorem:3dim}.
There exists a neighborhood $U\subset  \hom(\pi_1(\mathcal O^3),
\mathrm{SO}(3,1))$
of the  hyperbolic holonomy
such that for every $\rho\in U$
% in
%  $\rho\in \hom(\pi_1(\mathcal O^3),
% \mathrm{SL}_4(\mathbb R))$
$$
\dim T^{\mathrm{Zar}}_{[\rho]}
X( \mathcal O^3,\mathrm{SL}_4(\mathbb R))=\frac12
\sum_{i=1}^n
\dim
X(\partial_i \mathcal O^3, \mathrm{SL}_4(\mathbb R))
+ \mathsf{f},
$$
where $\mathsf{f}$ denotes the number of
components $\partial_i\mathcal O^3$ of $\partial \mathcal O^3$
 other than turnovers such that the image
 $\mathcal \rho(\pi_1( \partial_i\mathcal O^3))$
 is contained in a conjugate subgroup of
$\mathrm{SO}(2,1)$.
\end{Lemma}

\begin{proof}
 Infinitesimal projective rigidity is an open condition
 in the variety of representations $\hom(\pi_1(\mathcal O^3),\mathrm{SO}(3,1))$. This follows from semi-continuity,
 see for instance \cite[Lemma 3.2]{HeusenerPorti11}. So we may assume that
 infinitesimal projective rigidity holds true for any $\rho\in U$ and that,
 by Lemma~\ref{Lemma:resiso},
 $$
 \dim T^{\mathrm{Zar}}_{[\rho]}
X( \mathcal O^3,\mathrm{SL}_4(\mathbb R))= \dim H^1(\mathcal O^3,
\mathfrak{sl}_4(\mathbb R))=\frac12\sum_i \dim H^1(\partial_i\mathcal O^3,\mathfrak{sl}_4(\mathbb R)).
 $$
 Let $\rho_i$ denote the restriction of $\rho$ to $\pi_1( \partial_i\mathcal O^3  )$, we aim to relate
 the dimensions of $H^1(\partial_i\mathcal O^3,\mathfrak{sl}_4(\mathbb R))$ and of
 $X(\partial_i \mathcal O^3, \mathrm{SL}_4(\mathbb R))$. We distinguish several possibilities, again according to the topology of $\partial_i\mathcal O^3 $.

When  $\partial_i\mathcal O^3 $ is Euclidean,  by \cite[Theorem~4.11]{PortiDim}
$$
  \dim H^1(\partial_i\mathcal O^3,\mathfrak{sl}_4(\mathbb R)) =
\dim X(\partial_i \mathcal O^3, \mathrm{SL}_4(\mathbb R))
.$$

Assume next that $\partial_i\mathcal O^3 $ is hyperbolic and the image of
$\rho_i$ is not conjugate to a subgroup of $\mathrm{SO}(2,1)$.
Then by Lemma~\ref{Lemma:03irr} $\rho_i$ is irreducible as
a representation in  $\mathrm{SL}_4(\mathbb C)$. In fact  Lemma~\ref{Lemma:03irr}
is stated for hyperbolic holonomies, but the argument is algebraic. We just need to care
about being non-elementary, but this is an open property and we may chose $U$ so that $\rho_i$ is not elementary
(if  $\partial_i\mathcal O^3 $ is hyperbolic).
Then by Proposition~\ref{Proposition:Goldman}
$$
H^1(\partial_i\mathcal O^3,\mathfrak{sl}_4(\mathbb R))\cong
T^{\mathrm{Zar}}_{[\rho_i]}
X(\partial_i \mathcal O^3, \mathrm{SL}_4(\mathbb R))
%$$
%and $\dim
%T^{\mathrm{Zar}}_{[\rho_i]}
%X(\partial_i \mathcal O^3, \mathrm{SL}_4(\mathbb R)) = 
=\dim
X(\partial_i \mathcal O^3, \mathrm{SL}_4(\mathbb R))
.$$

Finally, assume  that $\partial_i\mathcal O^3 $ is hyperbolic and the image of
$\rho_i$ is conjugate to a subgroup of $\mathrm{SO}(2,1)$. By taking
$U$ small enough, $\rho_i$ is irreducible in  $\mathrm{SL}_3(\mathbb R)$
and by the results of Section~\ref{Section:DefSpace} we have:
$$
\dim H^1(\partial_i\mathcal O^3,\mathfrak{sl}_4(\mathbb R))=
\dim
X(\partial_i \mathcal O^3, \mathrm{SL}_4(\mathbb R))
+ 2.
$$
More precisely, in Theorem~\ref{Theorem:Slice} we construct an
analytic subset
$$
\mathcal S\subset H^1(\partial_i\mathcal O^3,\mathfrak{sl}_4(\mathbb R))
$$
with
$T^{Zar}_0\mathcal S= H^1(\partial_i\mathcal O^3,\mathfrak{sl}_4(\mathbb R))$,
that is an slice for the action by conjugation in the variety of representations.
This set is defined by the vanishing of an obstruction (Theorem~\ref{Theorem:Quadratic}),
so $\dim \mathcal S= \dim H^1(\partial_i\mathcal O^3,\mathfrak{sl}_4(\mathbb R))-1$. Furthermore,
the stabilizer of $\rho_i$ has dimension $1$, so by Corollary~\ref{Corollary:quotient},
$\dim X(\partial_i \mathcal O^3, \mathrm{SL}_4(\mathbb R))=\dim\mathcal S-1$.
% %
% % $\mathcal S\subset
% % \hom (\pi_1(\partial_i \mathcal O^3), \mathrm{SL}_4(\mathbb R))$ such that
% % $T^{\mathrm{Zar}}_{}
% (look refernence here )
\end{proof}

The following is another of the pieces in the proof of Theorem~\ref{Theorem:3dim}.

\begin{Proposition}
\label{Proposition:FuchsianSingular}
 Under the hypothesis of  Theorem~\ref{Theorem:3dim},
if some end of
$\mathrm{int}( \mathcal O^3 )$
is Fuchsian and   not a turnover,
then the character of $\rho$ is a singular point
of $X(\mathcal O^3, \mathrm{SL}_4(\mathbb R))$.
\end{Proposition}

\begin{proof}
 By Proposition~\ref{Proposition:NFD} and Lemma~\ref{Lemma:dimZTF},
 $\rho$ can be deformed so that the dimension of the
 Zariski tangent space strictly decreases.
\end{proof}

\subsection{Quadratic singularities}
\label{Section:Quadratic}

The following concludes the proof of  Theorem~\ref{Theorem:3dim}.

\begin{Proposition}
\label{Proposition:Quadratic}
 The singularity is   quadratic.
\end{Proposition}

\begin{proof}
To avoid technicalities, we work with the complexification,
namely the variety or representations of $\pi_1(\mathcal O^3)$ in
$\mathrm{SL}_4(\mathbb C)$ instead of
$\mathrm{SL}_4(\mathbb R)$. This is sufficient
to get quadratic singularities, by \cite[\S 3.3]{GoldmanMillson}.

As before, let
$$
\partial \mathcal O^3= \partial_1\mathcal O^3_1\cup \cdots
\cup  \partial_k\mathcal O^3
$$
denote the decomposition into connected components of the
boundary, and let
$\rho \colon \pi_1( \mathcal O^3)\to \mathrm{SO}(3,1)
\subset
\mathrm{SL}_4(\mathbb R)
$
denote the holonomy of the hyperbolic structure of  $\mathcal O^3$ and
$\rho_i\colon \pi_1(\partial_i \mathcal O^3)\to \mathrm{SL}_4(\mathbb R)$,
its restriction to $\partial_i \mathcal O^3$.

By Lemma~\ref{Lemma:03irr} $\rho$ is $\mathbb C$-irreducible,
and by Theorem~\ref{Theorem:slices}
there exists an analytic slice $\mathcal S$ at $\rho$ that satisfies:
% as in Theorem~\ref{Theorem:slices}.
$$T_\rho^{\mathrm{Zar}} \mathcal S\oplus B^1(\pi_1(\mathcal O^3),
\mathfrak{sl}_4(\mathbb C))=
Z^1(\pi_1(\mathcal O^3),
\mathfrak{sl}_4(\mathbb C)),
$$
hence $
T_\rho^{\mathrm{Zar}} \mathcal S\cong H^1(\pi_1(\mathcal O^3),
\mathfrak{sl}_4(\mathbb C))
$.
By \cite[Theorem~V.A.14]{GunningRossi}
any analytic subvariety of $\mathbb C^N$ is equivalent to a subvariety of a (non-singular)
$\mathbb C$-submanifold that has the same dimension  as the Zariski tangent space
of the analytic germ. So
there exist germs of \emph{non-singular}
$\mathbb C$-manifolds  $M$, $M_1,\ldots,M_k$, such that, after taking
a  small enough slice $\rho\in \mathcal S$ and choosing neighborhoods 
$ U_i\subset  \hom(\pi_1(\partial_ i\mathcal O^3),\mathrm{SL}_4(\mathbb C))$ 
of $\rho_i$
we have:
\begin{align*}
\mathcal S\subset M,  & \qquad T_\rho^{\mathrm{Zar}} \mathcal S=T_\rho M,  \\
  U_i\subset M_i, & \qquad T_{\rho\vert _{\pi_1\partial_ i\mathcal O^3} } \hom(\pi_1(\partial_ i\mathcal O^3),\mathrm{SL}_4(\mathbb C)) =T_{\rho\vert _{\pi_1\partial_ i\mathcal O^3}} M_i.
\end{align*}
The restriction map
$$
\mathcal
S\to \prod_i U_i\subset\prod_i \hom(\pi_1(\partial_ i\mathcal O^3),\mathrm{SL}_4(\mathbb C))
$$
is analytic, hence it extends to % \cite[Theorem~V.A.14]{GunningRossi} to
\begin{equation}
\label{eqn:phi}
\phi\colon M\to \prod_i M_i
\end{equation}
and since the restriction induces an injection
$$
0\to H^1(\mathcal O^3,\mathfrak{sl}_4(\mathbb C))
\to
\bigoplus_i
H^1(\partial_ i\mathcal O^3,\mathfrak{sl}_4(\mathbb C) )
$$
the map $\phi \colon M\to \prod_i M_i$  in \eqref{eqn:phi} is an
analytic immersion. The  following lemma implies that the singularity of $\mathcal S$ at $\rho$ is quadratic,
as each  $\hom(\pi_1(\partial_ i\mathcal O^3),\mathrm{SL}_4(\mathbb C))$ (or $U_i$)
is quadratic in $M_i$:

\begin{Lemma}
  $\phi(\mathcal S)=\phi(M)\cap \prod_i \hom(\pi_1(\partial_ i\mathcal O^3),\mathrm{SL}_4(\mathbb C))$.
\end{Lemma}

% %
% % is an analytic submanifold $S\subset M\subset \hom(\pi_1(\mathcal O^3),\mathrm{SL}_4(\mathbb C))$ (eg non-singular) such that
% % $T_\rho S=T_\rho M$.
% %
% % Again by \cite[Theorem~V.A.14]{GunningRossi}
% % there is a smooth manifold $M_i$ such that
% % $ \hom(\pi_1(\partial_ i\mathcal O^3).\mathrm{SL}_4(\mathbb R))\subset M_i$, and
% %
% %
% % GR:
% % $S\subset M$ analytic manifold, $TS=TM$
% % $ \hom(\pi_1(\partial_ i\mathcal O^3).\mathrm{SL}_4(\mathbb R))\subset M_i$
%
% $S\to \prod_i \hom(\pi_1(\partial_ i\mathcal O^3).\mathrm{SL}_4(\mathbb R))$
% extends to
% $M\to \prod_i M_i$ by analycity. It is an
% immersion (by inf proj rigidity)
%
% \begin{Lemma}
%  $S=M\cap \prod_i \hom(\pi_1(\partial_ i\mathcal O^3).\mathrm{SL}_4(\mathbb R))$
% \end{Lemma}

\begin{proof}[Proof of the lemma]
By construction we have the inclusion
 $$\phi(\mathcal S)\subseteq\phi(M)\cap \prod_i \hom(\pi_1(\partial_ i\mathcal O^3),\mathrm{SL}_4(\mathbb C))$$
and we prove equality
by contradiction. Assume
$$\phi(\mathcal S)\subsetneq\phi(M)\cap \prod_i \hom(\pi_1(\partial_ i\mathcal O^3),\mathrm{SL}_4(\mathbb C)).$$
Then for some $k\in\mathbb N$ there exists a $(k+1)$-jet in 
$$
c=c_0+c_1 t+\cdots + c_{k+1} t^{k+1}
\in J^{k+1}_{0,\phi(\rho)}\big(\mathbb R,   \phi(M)\cap \prod_i \hom(\pi_1(\partial_ i\mathcal O^3),\mathrm{SL}_4(\mathbb C) )\big)
$$
such that
$c$ is not a $(k+1)$-jet of $\phi(\mathcal{S})$,
$c \not\in
 J^{k+1}_{0,\phi(\rho)}(\mathbb R,   \phi(\mathcal{S}))$,
but
its $k$-th truncation is:
$$
[c]_k=c_0+c_1 t+\cdots + c_{k} t^{k}
\in J^{k}_{0,\phi(\rho)}(\mathbb R,   \phi(\mathcal{S})).
$$
Here we use local coordinates of the analytic variety $\prod_iM_i$ in an open
subset of $\mathbb{C}^N$, so that $c_0,\ldots,c_{k+1}\in\mathbb C^N$.

Since we assume that $[c]_k\in  J^{k}_{0,\phi(\rho)}(\mathbb R,   \phi(\mathcal{S}))$,  
we may write $[c]_k=\phi( \rho_k)$ 
for some $k$-jet $\rho_k\in J^k_{0, \rho}(\mathbb R, \mathcal S)$, in  $\mathcal S$,
that in its turn we write as
$$
\rho_k=\exp(a_1 t+\cdots a_k t^k)\rho \in J^{k+1}_{0,\rho} (\mathbb R, \mathcal S)
$$
for some maps (or 1-cochains) $a_i\colon\pi_1(\mathcal O^3) \to \mathfrak{sl}_4(\mathbb C)$.
By Lemma~\ref{Lemma:GoldmanObstructions},
the obstruction to extend $\rho_k$ to a $(k+1)$-jet
is a cohomology class in $H^2(\mathcal O^3,\mathfrak{sl}_4(\mathbb C))$.
By  Lemma~\ref{Lemma:resiso} the restriction map induces an isomorphism
$H^2(\mathcal O^3,\mathfrak{sl}_4(\mathbb C))\cong H^2(\partial\mathcal O^3,\mathfrak{sl}_4(\mathbb C))$.
Thus, by naturality of the obstruction and since $[c]_k$ is the truncation of the $(k+1)$-jet $c$
(in the product of varieties of representations of the components $\partial_i\mathcal O^3$), this obstruction vanishes
in $H^2(\mathcal{O}^3,\mathfrak{sl}(4,\mathbb C))$.
Therefore $\rho_k$ is the $k$-truncation of a $(k+1)$-jet
in the variety of representations:
$$
\rho_{k+1}=\exp(a_1 t+\cdots a_k t^k+a_{k+1} t^{k+1} ) \rho \in J^{k+1}_{0,\rho}
\big(\mathbb R, \hom(\pi_1 \mathcal O^3, \mathrm{SL}_4(\mathbb C) )\big).
$$
Next we want to conjugate $\rho_{k+1}$ so that the result is not only a jet in
$ \hom(\pi_1 \mathcal O^3, \mathrm{SL}_4(\mathbb C) )$ but in
the slice $\mathcal S$. For that purpose we use
Theorem~\ref{Theorem:slices}~(ii):
a neighborhood of $\rho$ in $ \hom(\pi_1 \mathcal O^3, \mathrm{SL}_4(\mathbb C) )$
is obtained by conjugating
the slice $\mathcal S$ by a neighborhood of the origin in $\mathrm{SL}_4(\mathbb C) $. Thus we write
$$
\rho_{k+1}=\operatorname{Ad}_{\exp (t^{k+1} b)} \rho_{k+1}'
$$
with $b\in\mathfrak{sl}_4(\mathbb C)$ and $\rho_{k+1}'   $  a jet in the slice $\mathcal S$
(whose $k$ truncation is $\rho_k$):
$$
\rho_{k+1}'=
\exp(a_1 t+\cdots a_k t^k+a_{k+1}' t^{k+1} ) \rho \in J^{k+1}_{0,\rho} (\mathcal{S}) .
$$
% By construction, the difference $a_{k+1}-a_{k+1}'$ is the inner cocycle corresponding
% to
% $b\in\mathfrak{sl}_4(\mathbb C)$, thus
% we consider
% $$\rho_{k+1}'=  \operatorname{Ad}_{\exp (-t^{k+1} b)}
% \rho_{k+1}.
% $$
%
% after conjugating by $\exp (-t^{k+1} b)$
% we may assume that $\rho_{k+1}$ is already a jet in the slice:
% $$
% \rho_{k+1}\in J^{k+1}_{0,\rho} (\mathcal{S}).
% $$
Notice that $\phi(\rho_{k+1}')$ and $c$ are two $(k+1)$-jets
in $\phi(M)$
whose $k$-truncations
are the same, hence their difference is $t^{k+1} v$ for some $v\in T_ {\phi(\rho)}\phi(M)$, namely
$v=\phi_*(u)$ for some $u\in T_\rho M$. As $T_\rho M=T_\rho S$, we use $u$   to modify the
$(k+1)$-the term of $\rho_{k+1}'$ and define:
$$
\rho_{k+1}''=\exp(t^{k+1} u)\rho_{k+1}'=\exp(a_1 t+\cdots a_k t^k+(a_{k+1}'+u) t^{k+1})\rho \in J^{k+1}_{0,\rho} (\mathcal{S}) .
$$
This $(k+1)$-jet satisfies $\phi( \rho_{k+1}'')= c$, hence a contradiction.
% % %
% % % % but $c(t) \not\in
% % % %  J^{k+1}_{0,\phi(\rho)}(\mathbb R,   \phi(\mathcal{S}))$
% % %
% % % To prove equality, we use Goldman obstruction theory.
% % % By  Lemma~\ref{Lemma:resiso} the restriction map induces an isomorphism
% % % $H^2(\mathcal O^3,\mathfrak g)\cong H^2(\partial\mathcal O^3,\mathfrak g)$.
% % % Hence,  each germ of analytic path in
% % % $\phi(M)\cap \prod_i \hom(\pi_1(\partial_ i\mathcal O^3),\mathrm{SL}_4(\mathbb C))$ is in fact a germ of analytic path in $\phi(\mathcal{S})$, thus the equality.
\end{proof}

By the properties of the slice in Theorem~\ref{Theorem:slices}, the singularity at the varieties of characters is quadratic.
\end{proof}

\begin{small}
\noindent \textsc{Departament de Matem\`atiques,  Universitat Aut\`onoma de Barcelona,
and Centre de Recerca Matem\`atica (UAB-CRM)\\
08193 Cerdanyola del Vall\`es, Spain }

\noindent \textsf{joan.porti@uab.cat}
\end{small}


\begin{thebibliography}{10}


\bibitem{Artin}
M.~Artin.
\newblock On the solutions of analytic equations.
\newblock {\em Invent. Math.}, 5:277--291, 1968.

\bibitem{BDL}
Samuel~A. Ballas, Jeffrey Danciger, and Gye-Seon Lee.
\newblock Convex projective structures on nonhyperbolic three-manifolds.
\newblock {\em Geom. Topol.}, 22(3):1593--1646, 2018.

\bibitem{Barbot}
Thierry Barbot.
\newblock Three-dimensional {A}nosov flag manifolds.
\newblock {\em Geom. Topol.}, 14(1):153--191, 2010.

\bibitem{BohmLafuente}
Christoph B\"{o}hm and Ramiro~A. Lafuente.
\newblock Real geometric invariant theory.
\newblock In {\em Differential geometry in the large}, volume 463 of {\em
  London Math. Soc. Lecture Note Ser.}, pages 11--49. Cambridge Univ. Press,
  Cambridge, 2021.

\bibitem{Brown}
Kenneth~S. Brown.
\newblock {\em Cohomology of groups}, volume~87 of {\em Graduate Texts in
  Mathematics}.
\newblock Springer-Verlag, New York-Berlin, 1982.

\bibitem{ChoiGoldman}
Suhyoung Choi and William~M. Goldman.
\newblock The deformation spaces of convex {$\mathbb{RP}^2$}-structures on
  2-orbifolds.
\newblock {\em Amer. J. Math.}, 127(5):1019--1102, 2005.

\bibitem{CLT1}
D.~Cooper, D.~D. Long, and M.~B. Thistlethwaite.
\newblock Flexing closed hyperbolic manifolds.
\newblock {\em Geom. Topol.}, 11:2413--2440, 2007.

\bibitem{CLT2}
Daryl Cooper, Darren Long, and Morwen Thistlethwaite.
\newblock Computing varieties of representations of hyperbolic 3-manifolds into
  {${\rm SL}(4,\mathbb R)$}.
\newblock {\em Experiment. Math.}, 15(3):291--305, 2006.

\bibitem{Goldman}
William~M. Goldman.
\newblock The symplectic nature of fundamental groups of surfaces.
\newblock {\em Adv. in Math.}, 54(2):200--225, 1984.

\bibitem{GoldmanMaryland}
William~M. Goldman.
\newblock Representations of fundamental groups of surfaces.
\newblock In {\em Geometry and topology ({C}ollege {P}ark, {M}d., 1983/84)},
  volume 1167 of {\em Lecture Notes in Math.}, pages 95--117. Springer, Berlin,
  1985.

\bibitem{GoldmanMillson}
William~M. Goldman and John~J. Millson.
\newblock The deformation theory of representations of fundamental groups of
  compact {K}\"{a}hler manifolds.
\newblock {\em Inst. Hautes \'{E}tudes Sci. Publ. Math.}, (67):43--96, 1988.

\bibitem{GunningRossi}
Robert~C. Gunning and Hugo Rossi.
\newblock {\em Analytic functions of several complex variables}.
\newblock Prentice-Hall Inc., Englewood Cliffs, N.J., 1965.
%
%
% \bibitem{HPpsl}
% Michael Heusener and Joan Porti.
% \newblock The variety of characters in {${\rm PSL}_2(\Bbb C)$}.
% \newblock {\em Bol. Soc. Mat. Mexicana (3)}, 10:221--237, 2004.

\bibitem{HeusenerPorti11}
Michael Heusener and Joan Porti.
\newblock Infinitesimal projective rigidity under {D}ehn filling.
\newblock {\em Geom. Topol.}, 15(4):2017--2071, 2011.

\bibitem{HeusenerPorti23}
Michael Heusener and Joan Porti.
\newblock The scheme of characters in {$\mathrm{SL}_2$}.
\newblock {\em Trans. Amer. Math. Soc.}, 376(9):6283--6313, 2023.


\bibitem{HPS}
Michael Heusener, Joan Porti, and Eva Su{\'a}rez~Peir{\'o}.
\newblock Deformations of reducible representations of 3-manifold groups into
  {${\rm SL}\sb{2}(\mathbf{C})$}.
\newblock {\em J. Reine Angew. Math.}, 530:191--227, 2001.

\bibitem{JohnsonMillson}
Dennis Johnson and John~J. Millson.
\newblock Deformation spaces associated to compact hyperbolic manifolds.
\newblock In {\em Discrete groups in geometry and analysis ({N}ew {H}aven,
  {C}onn., 1984)}, volume~67 of {\em Progr. Math.}, pages 48--106. Birkh\"auser
  Boston, Boston, MA, 1987.

\bibitem{KapovichBook}
Michael Kapovich.
\newblock {\em Hyperbolic manifolds and discrete groups}, volume 183 of {\em
  Progress in Mathematics}.
\newblock Birkh\"auser Boston Inc., Boston, MA, 2001.

\bibitem{LubotzkyMagid}
Alexander Lubotzky and Andy~R. Magid.
\newblock Varieties of representations of finitely generated groups.
\newblock {\em Mem. Amer. Math. Soc.}, 58(336):xi+117, 1985.

\bibitem{Luna}
D.~Luna.
\newblock Sur certaines op\'{e}rations diff\'{e}rentiables des groupes de
  {L}ie.
\newblock {\em Amer. J. Math.}, 97:172--181, 1975.

\bibitem{Parreau}
Anne Parreau.
\newblock Espaces de repr\'{e}sentations compl\`etement r\'{e}ductibles.
\newblock {\em J. Lond. Math. Soc. (2)}, 83(3):545--562, 2011.

\bibitem{PortiDim}
Joan Porti.
\newblock Dimension of representation and character varieties for two- and
  three-orbifolds.
\newblock {\em Algebr. Geom. Topol.}, 22(4):1905--1967, 2022.

\bibitem{RichardsonDuke}
R.~W. Richardson.
\newblock Conjugacy classes of {$n$}-tuples in {L}ie algebras and algebraic
  groups.
\newblock {\em Duke Math. J.}, 57(1):1--35, 1988.

\bibitem{RichardsonSlodowy}
R.~W. Richardson and P.~J. Slodowy.
\newblock Minimum vectors for real reductive algebraic groups.
\newblock {\em J. London Math. Soc. (2)}, 42(3):409--429, 1990.

\bibitem{Sikora}
Adam~S. Sikora.
\newblock Character varieties.
\newblock {\em Trans. Amer. Math. Soc.}, 364(10):5173--5208, 2012.

\bibitem{Simpson}
Carlos~T. Simpson.
\newblock Higgs bundles and local systems.
\newblock {\em Inst. Hautes \'{E}tudes Sci. Publ. Math.}, (75):5--95, 1992.

\bibitem{Weil}
Andr{\'e} Weil.
\newblock Remarks on the cohomology of groups.
\newblock {\em Ann. of Math. (2)}, 80:149--157, 1964.

\end{thebibliography}
\end{document}